\documentclass[a4paper,11pt,final]{article}

\usepackage{graphicx}
\usepackage{amssymb,amsmath,amsthm,mathrsfs,xfrac}
\usepackage{enumitem}  
\usepackage{a4wide}
\usepackage{hyperref}
\usepackage{array}
\usepackage{xcolor}
\usepackage{subcaption}
\usepackage[utf8]{inputenc}

\numberwithin{equation}{section}
\allowdisplaybreaks[4]

\newtheorem{proposition}{Proposition}[section]
\newtheorem{theorem}[proposition]{Theorem}
\newtheorem{lemma}[proposition]{Lemma}
\newtheorem{definition}[proposition]{Definition}
\newtheorem{corollary}[proposition]{Corollary}
\newtheorem{remark}[proposition]{Remark}

\renewenvironment{proof}{\smallskip\noindent\emph{\textbf{Proof.}}%
  \hspace{1pt}}{\hspace{-5pt}{\nobreak\quad\nobreak\hfill\nobreak%
    $\square$\vspace{2pt}\par}\smallskip\goodbreak}

\newenvironment{proofof}[1]{\smallskip\noindent{\textbf{Proof~of~#1.}}%
  \hspace{1pt}}{\hspace{-5pt}{\nobreak\quad\nobreak\hfill\nobreak%
    $\square$\vspace{2pt}\par}\smallskip\goodbreak}

\newcommand{\Id}{\mathinner{\mathrm{Id}}}
\newcommand{\pint}[1]{\mathaccent23{#1}}

\newcommand{\C}[1]{\mathbf{C}^{#1}}
\newcommand{\Cc}[1]{\mathbf{C}_c^{#1}}

\newcommand{\BV}{\mathbf{BV}}

\renewcommand{\L}[1]{{\mathbf{L}^#1}}
\newcommand{\Lloc}[1]{{\mathbf{L}_{\mathbf{loc}}^{#1}}}

\newcommand{\W}[2]{{\mathbf{W}^{#1,#2}}}

\newcommand{\modulo}[1]{{\left|#1\right|}}
\newcommand{\norma}[1]{{\left\|#1\right\|}}
\newcommand{\caratt}[1]{{\chi_{\strut#1}}}
\newcommand{\reali}{{\mathbb{R}}}
\newcommand{\naturali}{{\mathbb{N}}}

\renewcommand{\epsilon}{\varepsilon}
\renewcommand{\phi}{\varphi}
\renewcommand{\theta}{\vartheta}

\newcommand{\tv}{\mathinner{\rm TV}}
\newcommand{\spt}{\mathop{\rm spt}}
\newcommand{\sgn}{\mathop{\rm sgn}}

\renewcommand{\d}[1]{\mathinner{\mathrm{d}{#1}}}
\renewcommand{\div}{\mathinner{\nabla\cdot}} 

\makeatletter
\let\@fnsymbol\@arabic
\makeatother

\title{On a Hyperbolic-Parabolic Parasitoid-Parasite System:\\
  Well Posedness and Control}

\author{Rinaldo M.~Colombo\footnotemark[1]  \and Elena Rossi\footnotemark[2]}

\date{ }
\begin{document}
\maketitle
\footnotetext[1]{INdAM Unit, University of Brescia, via Branze, 38,
  25123 Brescia, Italy. \\ Email: \texttt{rinaldo.colombo@unibs.it} Orcid:~0000-0003-0459-585X }
\footnotetext[2]{Universit\`a degli Studi di Milano Bicocca,
  Dipartimento di Matematica e Applicazioni, via R.~Cozzi, 55, 20126
  Milano, Italy. E-mail: \texttt{elena.rossi@unimib.it} Orcid:~0000-0002-3565-4309}

\begin{abstract}

  \noindent We develop a time and space dependent predator -- prey
  model. The predators' equation is a non local hyperbolic balance
  law, while the diffusion of prey obeys a parabolic equation, so that
  predators \emph{``hunt''} for prey, while prey diffuse. A control
  term allows to describe the use of predators as parasitoids to limit
  the growth of prey--parasites.  The general well posedness and
  stability results here obtained ensure the existence of optimal pest
  control strategies, as discussed through some numerical
  integrations.

  The specific example we have in mind is that of \emph{Trichopria
    drosophil\ae} used to fight against the spreading of
  \emph{Drosophila suzukii}.

  \medskip

  \noindent\textit{2000~Mathematics Subject Classification:} 35L65,
  49J20, 35M30.

  \medskip


\end{abstract}

\section{Introduction}
\label{sec:Intro}

We consider the following mixed system on $\reali^n$
\begin{equation}
  \label{eq:1}
  \left\{
    \begin{array}{l}
      \partial_t u
      + \div \left(u \,v(t,w) \right)
      =
      f (t,x,w) \, u + q(t,x)
      \\
      \partial_t w
      - \mu \, \Delta w
      =
      g (t,x,u,w) \, w,
    \end{array}
  \right.
\end{equation}
where $u=u (t,x)$ and $w=w(t,x)$ represent respectively the predator
and the prey density at time $t \in \reali_+$ and position
$x\in\reali^n$. We remark that in the vector field $v$ the dependence
on the prey density $w$ is of a \emph{functional} nature thus
allowing, for instance, to describe predators that hunt for the prey
they perceive within a given distance. The parameter $\mu$, related to
the prey diffusion speed, is assumed to be strictly positive.

Once the fundamental well posedness and stability properties
for~\eqref{eq:1} are obtained, we consider the problem to steer the
solution to~\eqref{eq:1} to optimize a goal, typically represented by
the minimization of a functional defined on the solutions
to~\eqref{eq:1}. In the driving example we have in mind, the term $q$
in~\eqref{eq:1} represents the space and time dependent deployment of
parasitoids (predators) in the environment, aiming at limiting a given
parasites (prey). In other words, \eqref{eq:1} provides a possible
structure for the search for an optimal strategy in biological pest
control.  Preliminary general numerical results are provided
in~\cite{ElenaMBE}.

A specific situation that fits the present framework is the current
attempt to limit the spreading of \emph{Drosophila suzukii} (a pest
damaging fruits' cultivation) by means of \emph{ad hoc} deployments of
\emph{Trichopria drosophil\ae} (a parasitoid laying its eggs in the
larv\ae~of the \emph{Drosophila suzukii}), see~\cite{Crowder2007,
  Pfab2018489, RossiStacconi20199}. An obvious question risen by the
adoption of these biological strategies is the search for the optimal
time and space choices for the release of parasitoids in the
environment. The present paper offers a framework to test and compare
different strategies, see Section~\ref{sec:SNI}.

\medskip

From the analytic point of view, besides the introduction of the
control, the mixed system~\eqref{eq:1} comprehends the one studied
in~\cite{parahyp} also by taking into account general source terms
that may depend on the unknown variables, as well as on both $t$ and
$x$. Moreover, the flow $u\, v(t,w)$ in the first equation
in~\eqref{eq:1} accounts for the velocity chosen by predators in
response to the prey density distribution $w$. A key feature of the
mixed system~\eqref{eq:1} is the non locality and nonlinearity of the
function $v$ with respect to the prey density. For instance, the
choice
\begin{equation}
  \label{eq:2}
  \left(v (t,w)\right) (x)
  =
  \kappa (t,x) \,
  \frac{\nabla (w * \eta) (x)}{\sqrt{1+ \norma{\nabla (w * \eta) (x)}^2}},
\end{equation}
means that predators are directed towards regions where the
concentration of prey is greater.  Above, the positive function
$\kappa$ is the maximal speed of predators and may depend on time and
space. For any fixed positive smooth mollifier $\eta$, the
space-convolution product $\left(w (t) * \eta\right)(x)$ is an average
of the prey density at time $t$ around position $x$. The denominator
in~\eqref{eq:2} acts as a smooth normalisation factor.

\medskip

The next section is devoted to the well posedness and stability of the
Cauchy Problem for~\eqref{eq:1}. Then, we also deal with the optimal
control of the solutions to~\eqref{eq:1} by means of the control $q$
and aiming at the minimization of a given integral functional. A
specific application of these theoretical results is in
Section~\ref{sec:SNI}. All analytic proofs are deferred to
Section~\ref{sec:AP}.

\section{Main Results}
\label{sec:MR}

Below, we fix $T > t_o \geq 0$, possibly allowing the case
$T = +\infty$, and correspondingly we set
\begin{equation}
  \label{eq:J}
  I = [t_o,T] \mbox{ or } I = \left[t_o, +\infty\right[ \quad \mbox{ and } \quad
  J = \left\{(t_1,t_2) \in I^2 : t_1<t_2\right\}.
\end{equation}
The space dimension $n$ is fixed throughout, as well as the parameter
$\mu >0$. For the heat kernel we use the notation
$H_\mu (t,x) = (4 \, \pi \, \mu \, t)^{-n/2} \; \exp
\left(-\norma{x}^2 \middle/(4 \, \mu \, t)\right)$, where $t \in I$,
$x \in \reali^n$. As it is well known,
$\norma{H_\mu (t)}_{\L1 (\reali^n; \reali)} = 1$.

We recall below the definition of solution to~\eqref{eq:1}, slightly
extending that in~\cite{parahyp}, and adapting it to the present
setting of time and space dependent coefficients.

\begin{definition}
  \label{def:sol}
  A pair $(u,w) \in \C0 (I; \L1 (\reali^n; \reali^2))$ is a solution
  to problem~\eqref{eq:1} on $I$ if
  \begin{itemize}
  \item setting $a (t,x) = g\left(t,x,u (t,x),w (t,x)\right)$, $w$ is
    a weak solution to $\partial_t w - \mu \Delta w = a \,w$;

  \item setting $b (t,x) = f\left(t,x,w (t,x)\right)$ and
    $c (t,x) = \left(v (t,w (t))\right)\!(x)$, $u$ is a weak solution
    to $\partial_t u + \div (u \, c)= b \, u + q$.
  \end{itemize}
\end{definition}

\noindent The extension of Definition~\ref{def:sol} to
Cauchy problems is immediate. For completeness,
Definition~\ref{def:par} provides the definition of solution to the
parabolic equation $\partial_t w - \mu \Delta w = a \,w$, while
Definition~\ref{def:hyp} recalls the definition of solution to the
balance law $\partial_t u + \div (u \, c)= b \, u + q$.

Introduce the spaces
\begin{equation}
  \label{eq:X}
  \begin{aligned}
    \mathcal{U} & = (\L1 \cap \L\infty \cap \BV) (\reali^n; \reali)
    &\qquad \qquad & &\mathcal{U}^+ & = (\L1 \cap \L\infty \cap \BV)
    (\reali^n; \reali_+)
    \\
    \mathcal{X} & = \mathcal{U} \times \mathcal{U} & & &\mathcal{X}^+
    & = \mathcal{U}^+ \times \mathcal{U}^+
  \end{aligned}
\end{equation}
and the norm
\begin{equation}
  \label{eq:normX}
  \norma{(u,w)}_{\mathcal{X}}
  =
  \norma{u}_{\L1 (\reali^n; \reali)}
  +
  \norma{w}_{\L1 (\reali^n; \reali)}.
\end{equation}

We are now ready to state the key well posedness and stability result
of this paper.
\begin{theorem}
  \label{thm:main}
  Consider problem~\eqref{eq:1} under the following assumptions:
  \begin{enumerate}[label=$\boldsymbol{(v)}$]
  \item \label{v}
    $v \colon I \times (\L1\cap\L\infty) (\reali^n; \reali) \to (\C2
    \cap \W1\infty) (\reali^n; \reali^n)$ admits two maps
    $K_v \in \Lloc\infty (I; \reali_+)$ and
    $C_v \in \Lloc\infty (I\times\reali_+; \reali_+)$ weakly
    increasing in each argument and such that, for all $t \in I$ and
    $w,w_1, w_2 \in (\L1\cap\L\infty) (\reali^n; \reali)$,
    \begin{align*}
      \norma{v (t,w)}_{\L\infty (\reali^n; \reali^n)}
      \leq \
      & K_v (t) \, \norma{w}_{\L1 (\reali^n; \reali)},
      \\
      \norma{\nabla v (t,w)}_{\L\infty (\reali^n;\reali^{n\times n})}
      \leq \
      & K_v (t) \, \norma{w}_{\L\infty(\reali^n; \reali)},
      \\
      \norma{v (t,w_1) - v (t,w_2)}_{\L\infty (\reali^n; \reali^n)}
      \leq \
      & K_v (t) \, \norma{w_1 - w_2}_{\L1(\reali^n; \reali)},
      \\
      \norma{\nabla \left(\div v (t,w)\right)}_{\L1 (\reali^n; \reali^n)}
      \leq \
      & C_v \! \left(t, \norma{w}_{\L1(\reali^n; \reali)}\right)
        \, \norma{w}_{\L1(\reali^n; \reali)},
      \\
      \norma{\div\left(v (t,w_1) - v (t,w_2)\right)}_{\L1 (\reali^n; \reali)}
      \leq \
      & C_v \! \left(t, \norma{w_2}_{\L\infty(\reali^n; \reali)}\right)
        \, \norma{w_1 - w_2}_{\L1(\reali^n; \reali)}.
    \end{align*}
  \end{enumerate}

  \begin{enumerate}[label=$\boldsymbol{(f)}$]
  \item \label{f}
    $f \colon I \times \reali^n \times \reali \to \reali^n$ admits a
    weakly increasing map $K_f \in \Lloc\infty (I; \reali_+)$ such
    that, for a.e.~$t \in I$, all $w_1,w_2 \in \reali_+$ and all
    $w \in \BV (\reali^n; \reali)$,
    \begin{align*}
      \sup_{x \in \reali^n} \modulo{f (t,x,w_1) - f (t,x,w_2)}
      \leq \
      & K_f (t) \; \modulo{w_1 - w_2}\,,
      \\
      \sup_{x \in \reali^n } f (t,x,w_1)
      \leq \
      & K_f (t)  \, (1+w_1) \,,
      \\
      \tv f\left(t, \cdot, w (\cdot)\right)
      \leq \
      & K_f (t)
        \left(
        1
        +
        \norma{w}_{\L\infty (\reali^n; \reali)}
        +
        \tv (w)
        \right) \,.
    \end{align*}

  \end{enumerate}

  \begin{enumerate}[label=$\boldsymbol{(g)}$]
  \item \label{g}
    $g \colon I \times \reali^n \times \reali \times \reali \to
    \reali$ admits a weakly increasing map
    $K_g \in \Lloc\infty (I; \reali_+)$ such that, for a.e.~$t \in I$
    and all $u_1,u_2,w_1,w_2 \in \reali_+$,
    \begin{align*}
      \sup_{x \in \reali^n} \modulo{g (t,x,u_1,w_1) - g (t,x,u_2,w_2)}
      \leq \
      & K_g (t) \; \left(\modulo{u_1-u_2} + \modulo{w_1-w_2}\right) ,
      \\
      \sup_{(x,u,w) \in \reali^n \times \reali_+ \times \reali_+}
      g (t,x,u,w)
      \leq \
      & K_g (t).
    \end{align*}
  \end{enumerate}

  \begin{enumerate}[label={$\boldsymbol{(q)}$}]
  \item \label{q*}
    $q \in \L\infty(I \times \reali^n; \reali_+) \cap \L\infty (I; \L1
    (\reali^n; \reali_+))$ and $q (t) \in \BV (\reali^n; \reali_+)$,
    for $t \in I$.
  \end{enumerate}

  \noindent Then, for any initial datum $(u_o,w_o) \in \mathcal{X}^+$,
  problem~\eqref{eq:1} admits a unique solution
  \begin{displaymath}
    (u, w) \in \C0 (I, \L1 (\reali^n; \reali_+^2))
  \end{displaymath}
  in the sense of Definition~\ref{def:sol} and, moreover,
  \begin{enumerate}[label=\bf{(\arabic*)}]
  \item \textbf{\emph{A priori} estimates:} for all $t \in I$, we have
    \begin{align*}
      \norma{w (t)}_{\L1 (\reali^n)}
      \leq \
      & \norma{w_o}_{\L1 (\reali^n)} \; e^{K_g (t) \, (t-t_o)},
      \\
      \norma{w (t)}_{\L\infty (\reali^n)}
      \leq \
      & \norma{w_o}_{\L\infty (\reali^n)} \; e^{K_g (t) \, (t-t_o)},
      \\
      \norma{u (t)}_{\L1 (\reali^n)}
      \leq \
      & \left(\norma{u_o}_{\L1 (\reali^n)}
        +\norma{q}_{\L1 ([t_o,t]\times\reali^n)}\right)
      \\
      & \qquad
        \times
        \exp \left[
        K_f (t) \, (t-t_o)
        \left(
        1
        +
        \norma{w_o}_{\L\infty (\reali^n)} e^{K_g (t) \, (t-t_o)}
        \right)
        \right],
      \\
      \norma{u (t)}_{\L\infty (\reali^n)}
      \leq \
      & \left( \norma{u_o}_{\L\infty (\reali^n)}
        + \norma{q}_{\L1 ([t_o,t]; \L\infty(\reali^n))}\right)
      \\
      & \qquad
        \times
        \exp \left[
        \left(K_f (t) + K_v (t)\right) (t-t_o)
        \left(
        1
        +
        \norma{w_o}_{\L\infty (\reali^n)} e^{K_g (t) \, (t-t_o)}
        \right)
        \right] .
    \end{align*}

  \item \textbf{Lipschitz continuous dependence on the initial data:}
    for $(u_o,w_o), \, (\tilde u_o, \tilde w_o) \in \mathcal{X}^+$,
    \begin{equation}
      \label{eq:37}
      \norma{(u (t) , w(t)) - (\tilde u (t), \tilde w (t))}_{\mathcal{X}}
      \leq \mathcal{C}_o (t,r) \, \norma{(u_o, w_o) - (\tilde u_o, \tilde
        w_o)}_{\mathcal{X}}
    \end{equation}
    where the locally bounded function $\mathcal{C}_o$ is defined
    in~\eqref{eq:36} and $r$ is an upper bound for the $\L1$ norm, the
    $\L\infty$ norm and the total variation of the initial data,
    see~\eqref{eq:38}.

  \item \textbf{Stability with respect to the control $q$:} for all
    $q, \tilde q$ satisfying~\textbf{(q*)}, for all $t \in I$,
    \begin{equation}
      \label{eq:27}
      \norma{(u (t) , w(t)) - (\tilde u (t), \tilde w (t))}_{\mathcal{X}}
      \leq
      \mathcal{C}_q (t,r) \,
      \norma{q - \tilde q}_{\L1 ([t_o,t] \times \reali^n)} ,
    \end{equation}
    where the locally bounded function $\mathcal{C}_q$ is defined
    in~\eqref{eq:33} and $r$ is an upper bound for the $\L1$ norm, the
    $\L\infty$ norm and the total variation of the initial data,
    see~\eqref{eq:38}.
  \end{enumerate}
\end{theorem}

\noindent To prove Theorem~\ref{thm:main}, following the general lines
of~\cite{parahyp}, we study separately, but symmetrically, the
parabolic and the hyperbolic problems that constitute~\eqref{eq:1},
namely
\begin{displaymath}
  \partial_t w - \mu \, \Delta w = a (t,x) \,w
  \qquad \mbox{ and } \qquad
  \partial_t u +\div (c (t,x) \, u) = b (t,x) \, u + q (t,x) \,.
\end{displaymath}
with $a,b$ and $c$ as in Definition~\ref{def:sol}. All estimates use
exclusively the $\L1$ or $\L\infty$ norms and the total variation in
space.

\begin{remark}
  Note the different behaviors of $f$ and $g$ allowed by
  conditions~\ref{f} and~\ref{g}, namely
  $\sup_{x \in \reali^n } f (t,x,w) \leq K_f (t) \, (1+w)$ and
  $\sup_{(x,u,w) \in \reali^n \times \reali_+ \times \reali_+} g
  (t,x,u,w) \leq K_g (t)$.  For instance, $f$ may well increase in
  $w$, while $g$ may decrease in both $u$ and $w$. Thus, the classical
  Lotka-Volterra source terms $f(w) = \alpha \, w - \beta$ and
  $g(u) = \gamma - \delta \, u$ (for $\alpha, \beta, \gamma, \delta$
  positive and constant) are compatible with~\ref{f} and~\ref{g},
  comprising the problem studied in~\cite{parahyp} when $q \equiv 0$.
\end{remark}

\bigskip

Theorem~\ref{thm:main} allows to consider optimal control problems
based on~\eqref{eq:1}. To this aim, introduce a cost functional
measuring the relevance of the presence of the pest, for instance
quantifying its effect on cultivation. Inspired
by~\cite[\S~4]{ElenaMBE}, we propose a cost of the general form
\begin{equation}
  \label{eq:34}
  \mathcal{I}
  =
  \int_I \int_{\reali^n}
  \Phi \left(t,x, u (t,x), w (t,x)\right) \d{x} \d{t} \,.
\end{equation}
It is clear that various assumptions on the function $\Phi$ ensure
that the integral on the right hand side of~\eqref{eq:34} is a
continuous function of $(u,w)$ in $\mathcal{X}$. Therefore,
\textbf{(3)} in~Theorem~\ref{thm:main} ensures that $\mathcal{I}$ is a
continuous function of the control $q$ in $\L1$.

In practice, the choice of a real strategy depends on a finite set of
parameters, say $p \in \reali^m$, defining, for instance, the
(time/space) support of $q$, or the maximal value of $q$, or its
(time/space) integral. We are thus lead to minimize a compositions of
maps of the type
\begin{displaymath}
  \begin{array}{ccccccc}
    \reali^m
    & \to
    & \L\infty \left(I; \L1 (\reali^n; \reali)\right)
    & \to
    & \mathcal{X}^+
    & \to
    & \reali
    \\
    p
    & \to
    & q
    & \to
    & (u,w)
    & \to
    & \mathcal{I}
  \end{array}
\end{displaymath}
to which, thanks to Theorem~\ref{thm:main}, {Weierstra\ss} Theorem can
be applied, ensuring the existence of an optimal strategy $p_*$. The
actual computation of $p_*$ can be achieved through standard numerical
procedures dedicated to the optimization of Lipschitz continuous
functions. The next section is devoted to specific examples.

\section{Optimized Timing of Parasitoids' Releases}
\label{sec:SNI}

We present below a sample of the possible behaviors of solutions
to~\eqref{eq:1}. Further examples can be found in~\cite{ElenaMBE}.

Inspired by~\cite{Pfab2018489, RossiStacconi20199}, we address the
problem of optimizing the timing and the location of parasitoids'
(=predators') releases in the case of a parasite (=prey) whose
reproduction is seasonal and geographically localized. To this aim, we
consider the following instance of~\eqref{eq:1} in the case of $n = 2$
space dimensions
\begin{equation}
  \label{eq:35}
  \left\{
    \begin{array}{l}
      \partial_t u
      + \div \left(u \,v(w) \right)
      =
      (\alpha \, w - \beta) u + q(t,x)
      \\
      \partial_t w
      - \mu \, \Delta w
      =
      \left(
      \gamma \, (1-\sin t) \, \caratt{B} (x)
      \left(      1
      -
      \dfrac{w}{C}
      \right)
      -
      \delta \, u
      \right)
      \, w,
    \end{array}
  \right.
\end{equation}
Here, as usual, $t$ is time and $x$ is the space coordinate in
$\reali^2$. Moreover, $\alpha\, w$ is the predator natality due to
predation, $\beta$ is the predators' mortality, $\delta$ is the prey
mortality due to predation and $C$ is the prey carrying capacity. The
prey natality\footnote{$\caratt{B}$ is the characteristic function of
  the set $B$: $\caratt{B} (x)=1 \iff x \in B$ and
  $\caratt{B} (x) =0 \iff x \in \reali^n \setminus B$.}
$\gamma \, (1-\sin t) \, \caratt{B} (x)$ is \emph{seasonal}, i.e.~it
is $2\pi$--periodic in time, and \emph{localized}, i.e.~it is
supported in the ball $B$ centered at the origin with radius $2$. The
speed $v$ is chosen as in~\eqref{eq:2}, with $\kappa$ constant. The
parasitoids predate hunting for parasites in the direction of the
highest average prey density gradient within a radius $\ell$, which
hence measures the predator horizon. We summarize here the choices of
functions and parameters in~\eqref{eq:35}--\eqref{eq:2}, apart from
$q$ to be chosen below:
\begin{equation}
  \label{eq:40}
  \begin{array}{@{}r@{\,}c@{\,}l@{\quad}r@{\,}c@{\,}l@{\quad}r@{\,}c@{\,}l}
    \alpha
    & =
    & 0.25
    & \beta
    & =
    & 2.00
    & \gamma
    & =
    & 9.00
    \\
    \delta
    & =
    & 0.50
    & C
    & =
    & 10.0
    & \ell
    & =
    & 0.80
    \\
    \kappa
    & =
    & 2.00
  \end{array}
  \qquad\quad
  \eta (x) =
  \left\{
    \begin{array}{l@{\qquad}r@{\,}c@{\,}l}
      \dfrac{4}{\pi \, \ell^2} \,
      \left(1- \dfrac{\norma{x}^2}{\ell^2}\right)^3
      & \norma{x}
      & \leq
      & \ell\,,
      \\
      0
      & \norma{x}
      & >
      & \ell \,.
    \end{array}
  \right.
\end{equation}
We now seek strategies $q = q (t,x)$ to release parasitoids so that
the parasite population is kept small in the rectangle
$R = [1,3] \times [-3,3]$, which we assume is the region where the
presence of the parasites is most harmful. The regions $B$ and $R$ are
chosen so that they are different but overlapping. Thus, for
simplicity, we aim at the minimization of
\begin{equation}
  \label{eq:39}
  \mathcal{I} = \int_{4\pi}^{12\pi} \int_R w (t,x) \d{x} \, \d{t} \,,
\end{equation}
although within the present framework~\eqref{eq:34} more complex costs
can be considered. Another natural choice, for instance, might be the
minimization of the pest population $w$ only in specific periods,
e.g.~when fruits are ripening on the trees, as in the case of the
\emph{Drosophila suzukii}.  As initial datum we choose
\begin{equation}
  \label{eq:41}
  u_o (x) \equiv 0\,,\qquad w_o (x) = 2 \, \caratt{B} (x) \,.
\end{equation}
Clearly, Theorem~\ref{thm:main} applies
to~\eqref{eq:35}--\eqref{eq:2}--\eqref{eq:40}--\eqref{eq:41} and the
cost~\eqref{eq:39} fits into~\eqref{eq:34}.

In the examples below, we use the Lax--Friedrichs
scheme~\cite[\S~12.5]{LeVequeBook2002} to integrate the hyperbolic
convective term and an explicit finite difference algorithm to deal
with the parabolic equation. Furthermore, we exploit dimensional
splitting~\cite[\S~19.5]{LeVequeBook2002} and a further splitting to
take care of the source terms~\cite[\S~17.1]{LeVequeBook2002}, which
are computed through a second order Runge--Kutta method (corresponding
to $\alpha = 1/2$ in~\cite[\S~12.5, p.~327]{Num}). Refer
to~\cite{BurgerChowellMuletVillada1, BurgerChowellMuletVillada2,
  RS2016} for alternative algorithms. The numerical domain is the
rectangle $[-4-\ell,4+\ell]\times[-4-\ell,4+\ell]$ and we let the
parameters $\alpha, \beta$ and $\gamma$ vanish outside the physical
domain $[-4,4]\times[-4,4]$. The computations below were obtained with a uniform mesh consisting of $2^{10} \times 2^{10}$ points.

\medskip

First, as a reference case, we
integrate~\eqref{eq:35}--\eqref{eq:2}--\eqref{eq:40}--\eqref{eq:41}
with $q \equiv 0$. The results are displayed in Figure~\ref{fig:u0}.
\begin{figure}[!ht]
  \includegraphics[width=0.4\linewidth, trim= 419 90 20
  40,clip=true]{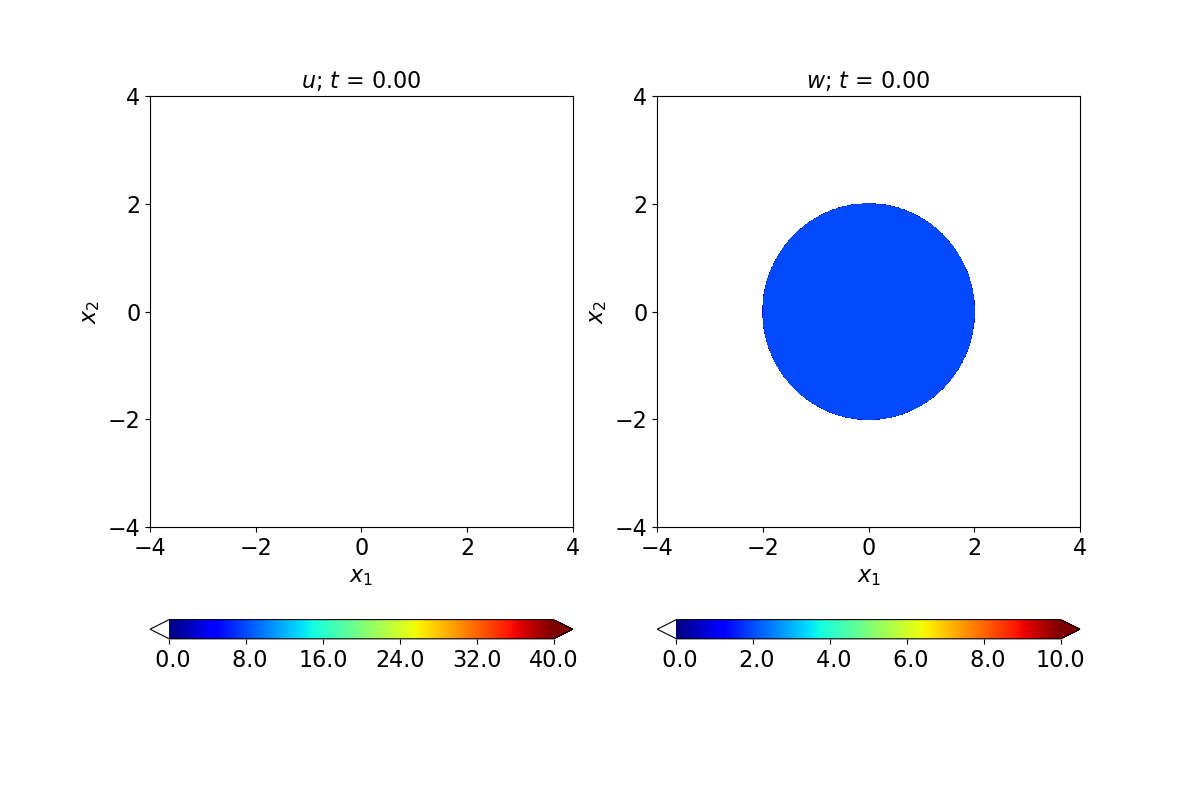} %
  \includegraphics[width=0.6\linewidth, trim= 20 12 40
  20,clip=true]{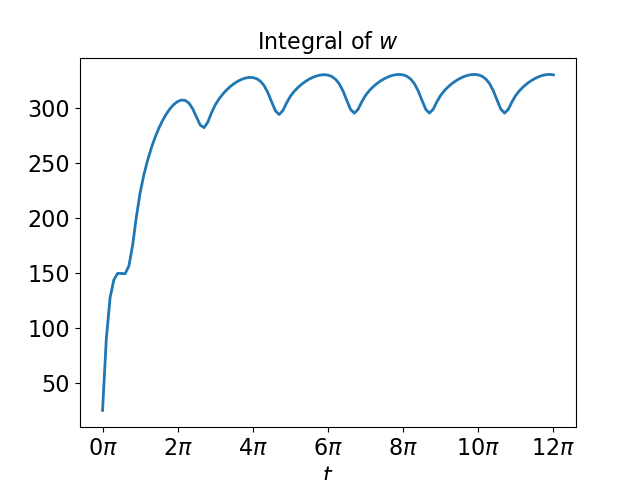}
  \caption{Left, the initial datum~\eqref{eq:41} for $w$ in the
    $x$--plane and, right, the total amount of parasites
    $\int_{[-4,4]^2} w (t,x) \d{x}$ on the whole physical domain as a
    function of time.}
  \label{fig:u0}
\end{figure}
Since parasitoids are absent, parasites evolve with a logistic growth
with capacity $C$ and a $2\pi$--periodic natality. After two periods,
the total number of parasites is approximately time periodic, with a
high mean value.

We now assume that at time $4\pi$ measures need to be taken to reduce
the presence of parasites. This is achieved through the release in the
environment of the parasitoid $u$, which is described by the function
$q$ in~\eqref{eq:35}. Different strategies correspond to different
choices of $q$. The ones we consider below differ both in the space
and time dependence: they may take place in the ball $B$ where the
parasites are born, or on the rectangle $R$ where parasites are
harmful. Moreover, they can take place uniformly in time (on
$I_0 = [4\pi, 12\pi]$) or in the time intervals where parasites are
more ($I_1 = \sin^{-1} ([-1, -1/\sqrt2]) \cap I_0$), middle
($I_2 = \cos^{-1} ([-1, -1/\sqrt2]) \cap I_0$) or less
($I_3 = \sin^{-1} ([1/\sqrt2, 1]) \cap I_0$) prolific, see
Figure~\ref{fig:gq}.
\begin{figure}[!h]
  \centering \includegraphics[width=0.25\linewidth]{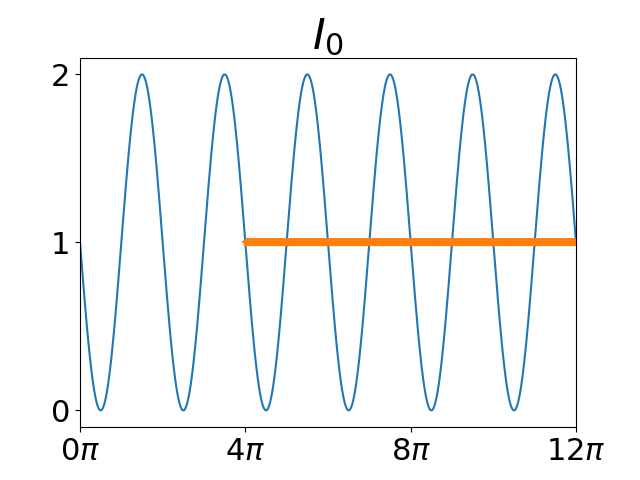}%
  \includegraphics[width=0.25\linewidth]{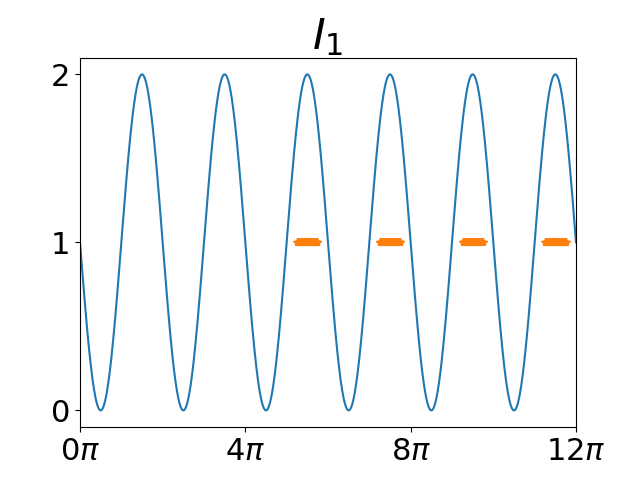}%
  \includegraphics[width=0.25\linewidth]{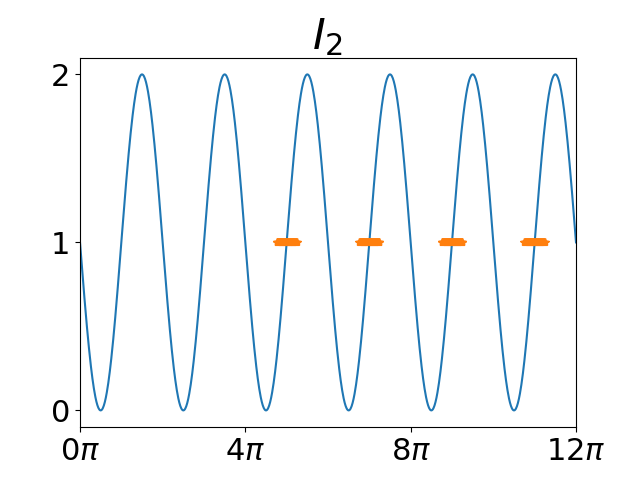}%
  \includegraphics[width=0.25\linewidth]{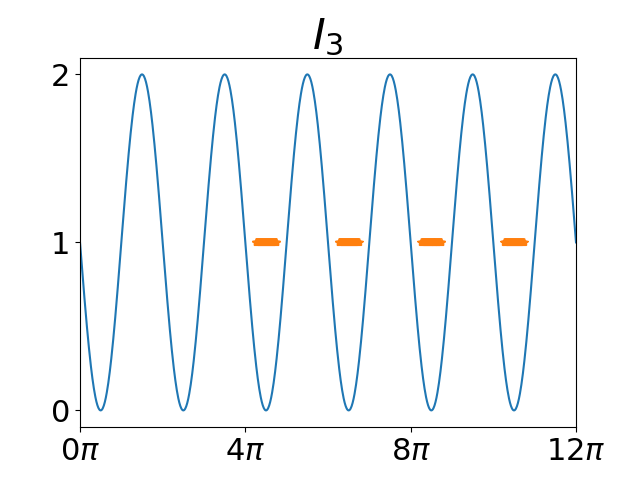}\\
  \caption{Characteristic functions of the time intervals, from left
    to right, $I_0$, $I_1$, $I_2$ and $I_3$ used in the definitions of
    the controls~\eqref{eq:16}, plotted together with the map
    $t \to 1-\sin t$ appearing in the natality of the parasite
    in~\eqref{eq:35}.}
  \label{fig:gq}
\end{figure}
These strategies correspond to the following choices of $q$:
\begin{equation}
  \label{eq:16}
  \begin{array}{r@{\,}c@{\,}l@{\qquad\qquad}r@{\,}c@{\,}l}
    q^B_0 (t)
    & =
    & 3.166287 \, \caratt{I_0} (t) \, \caratt{B} (x)
    & q^R_0 (t)
    & =
    & 3.315728 \, \caratt{I_0} (t)\, \caratt{R} (x)
    \\
    q^B_1 (t)
    & =
    & 12.66515 \, \caratt{I_1} (t)\, \caratt{B} (x)
    & q^R_1 (t)
    & =
    & 13.26291 \, \caratt{I_1} (t)\, \caratt{R} (x)
    \\
    q^B_2 (t)
    & =
    & 12.66515 \, \caratt{I_2} (t)\, \caratt{B} (x)
    & q^R_2 (t)
    & =
    & 13.26291 \, \caratt{I_2} (t)\, \caratt{R} (x)
    \\
    q^B_3 (t)
    & =
    & 12.66515 \, \caratt{I_3} (t) \, \caratt{B} (x)
    & q^R_3 (t)
    & =
    & 13.26291\, \caratt{I_3} (t)\, \caratt{R} (x) \,.
  \end{array}
\end{equation}
The above values are chosen so that the amount of parasitoids inserted
in the environment is constant, i.e.
\begin{displaymath}
  \int_0^{12\pi} \int_{\reali^2} q_i^A (t,x) \d{x} \d{t} = 1000\,
  \quad \mbox{ for } i = 0, 1, 2, 3 \mbox{ and } A = B,R \,.
\end{displaymath}

The numerical integrations
of~\eqref{eq:35}--\eqref{eq:2}--\eqref{eq:40}--\eqref{eq:41} with the
controls~\eqref{eq:16} yield the following values for the
cost~\eqref{eq:39}:
\begin{center}
  \begin{tabular}{c|c|c|c|c}
    $\mathcal{I}$
    & 0
    & 1
    & 2
    & 3
    \\ \hline
    $B$
    & 1179.05
    & 1318.74
    & 1332.75
    & 1232.41
    \\ \hline
    $R$
    & 874.420 
    & 1068.13
    & 1098.85
    & 1080.19
  \end{tabular}
  \qquad\qquad when $q \equiv 0$, $\mathcal{I} = 1866.98$.
\end{center}
In the different cases of the controls in~\eqref{eq:16}, the
instantaneous costs $t \to \int_R w (t,x) \, \d{x}$ are displayed in
Figure~\ref{fig:costs}. All solutions
to~\eqref{eq:35}--\eqref{eq:2}--\eqref{eq:40}--\eqref{eq:41} show a
somewhat periodic behavior for $t > 4\pi$.
\begin{figure}[!h]
  \centering
  \begin{subfigure}[t]{0.3\textwidth}
    \includegraphics[width=\linewidth, trim = 30 15 30 37, clip =
    true]{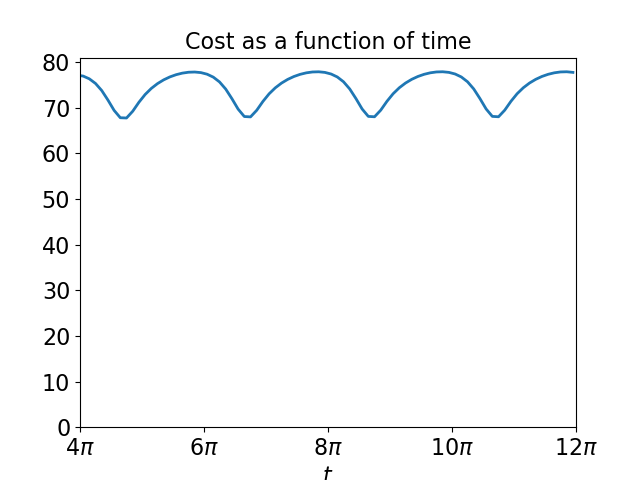}
    \caption{No control: $q\equiv 0$.}
    \label{fig:u0c}
    \vspace*{2mm}
  \end{subfigure}%
  ~
  \begin{subfigure}[t]{0.3\textwidth}
    \includegraphics[width=\linewidth, trim = 30 15 30 37, clip =
    true]{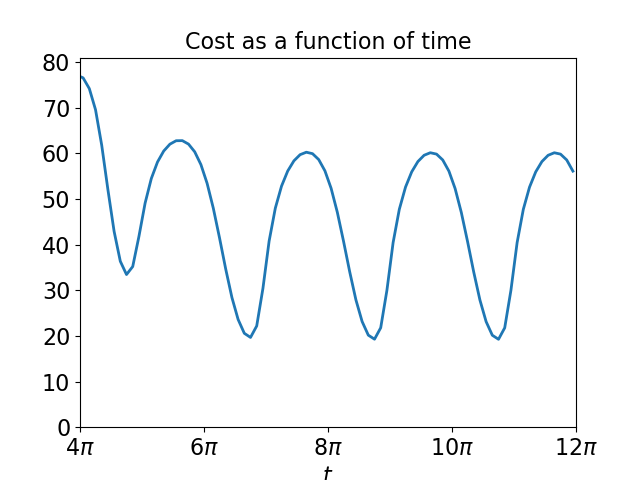}
    \caption{$q_0^B$}
    \label{fig:q}
    \vspace*{2mm}
  \end{subfigure}%
  ~
  \begin{subfigure}[t]{0.3\textwidth}
    \includegraphics[width=\linewidth, trim = 30 15 30 37, clip =
    true]{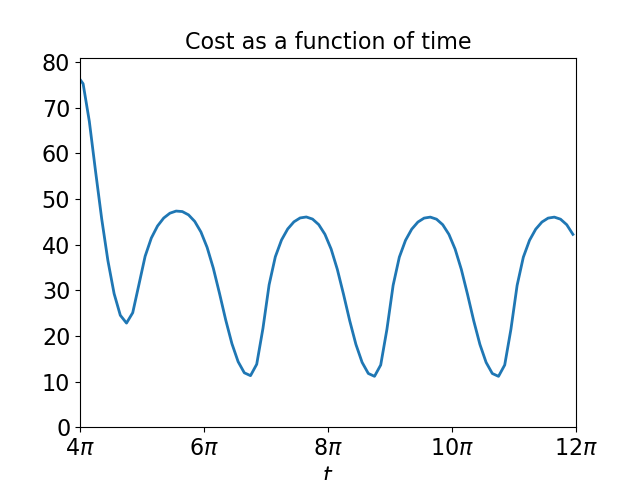}
    \caption{$q_0^R$}
    \label{fig:qr}
    \vspace*{2mm}
  \end{subfigure}\\
  \begin{subfigure}[t]{0.3\textwidth}
    \includegraphics[width=\linewidth, trim = 30 15 30 37, clip =
    true]{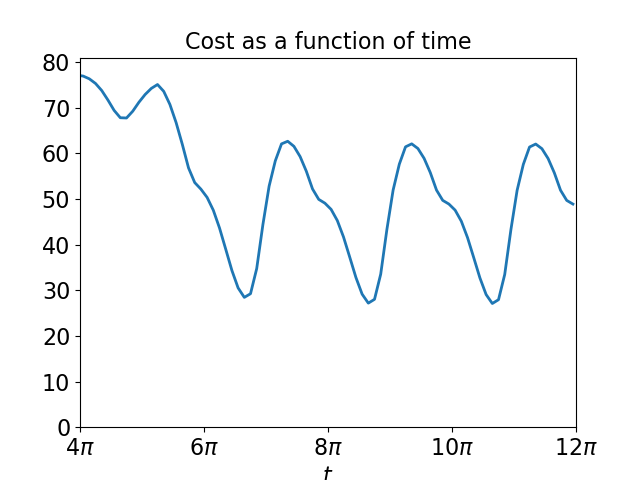}
    \caption{$q_1^B$}
    \label{fig:q1}
    \vspace*{2mm}
  \end{subfigure}%
  ~
  \begin{subfigure}[t]{0.3\textwidth}
    \includegraphics[width=\linewidth, trim = 30 15 30 37, clip =
    true]{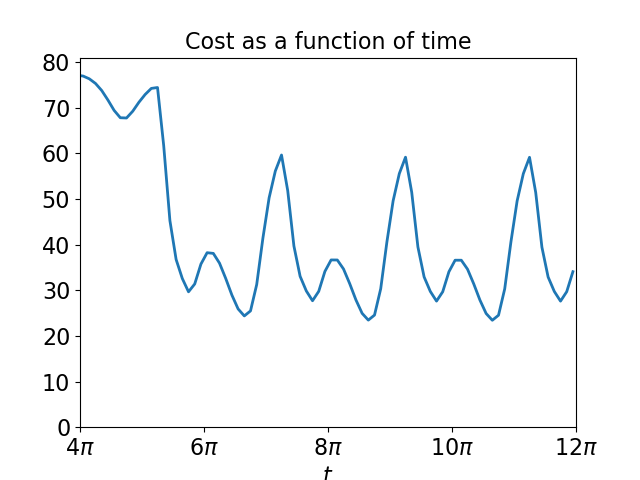}
    \caption{$q_1^R$}
    \label{fig:q1r}
    \vspace*{2mm}
  \end{subfigure}%
  ~
  \begin{subfigure}[t]{0.3\textwidth}
    \includegraphics[width=\linewidth, trim = 30 15 30 37, clip =
    true]{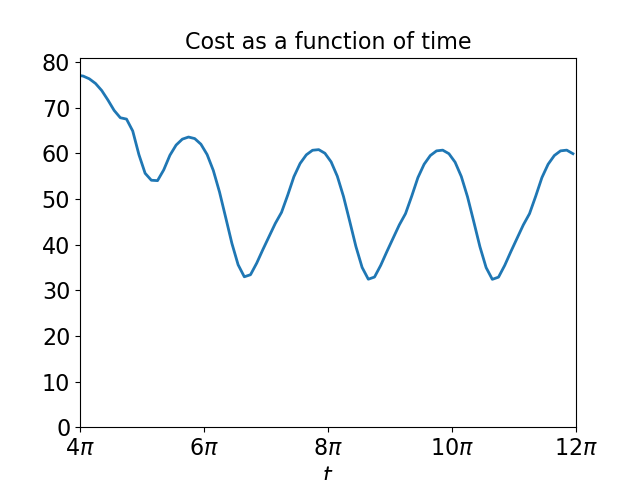}
    \caption{$q_2^B$}
    \label{fig:q3}
    \vspace*{2mm}
  \end{subfigure}\\
  \begin{subfigure}[t]{0.3\textwidth}
    \includegraphics[width=\linewidth, trim = 30 15 30 37, clip =
    true]{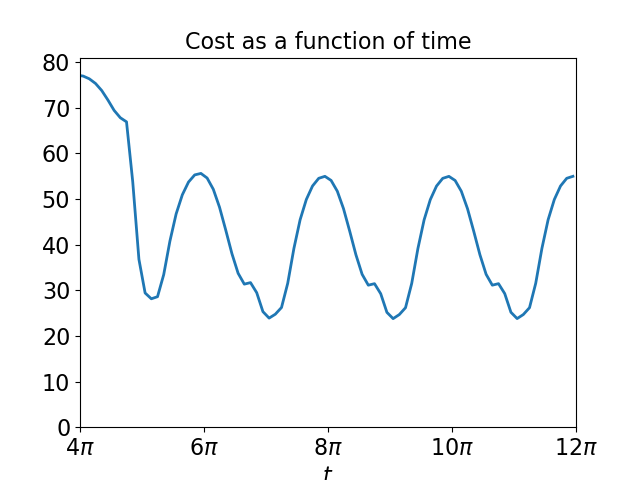}
    \caption{$q_2^R$}
    \label{fig:q3r}
  \end{subfigure}%
  ~
  \begin{subfigure}[t]{0.3\textwidth}
    \includegraphics[width=\linewidth, trim = 30 15 30 37, clip =
    true]{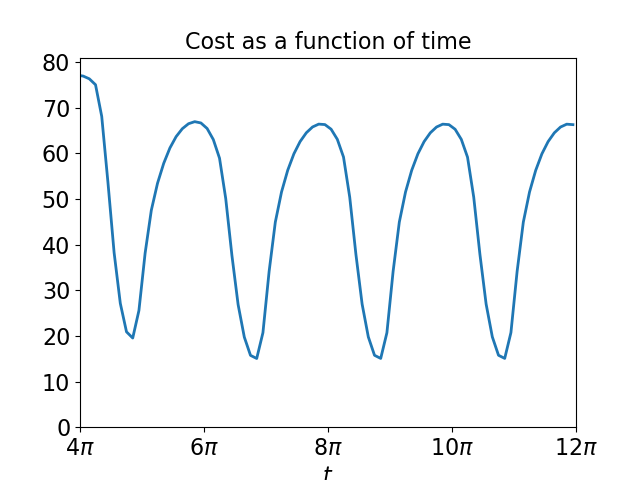}
    \caption{$q_3^B$}
    \label{fig:q2}
  \end{subfigure}%
  ~
  \begin{subfigure}[t]{0.3\textwidth}
    \includegraphics[width=\linewidth, trim = 30 15 30 37, clip =
    true]{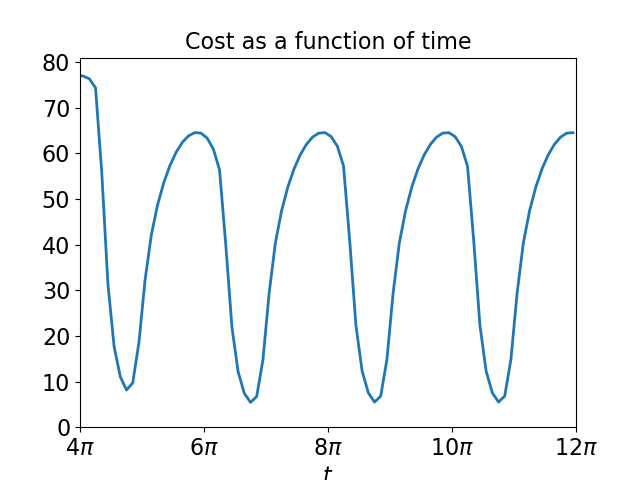}
    \caption{$q_3^R$}
    \label{fig:q2r}
  \end{subfigure}\\
  \caption{Graphs of the instantaneous cost
    $t \to \int_R w (t,x) \d{x}$ corresponding to the
    controls~\eqref{eq:16} on the time interval $[4\pi,
    12\pi]$. Figure~\ref{fig:u0c} corresponds to the diffusion of
    parasites with no control. The most effective strategy, in the
    sense it minimizes~\eqref{eq:39}, is in Figure~\ref{fig:qr}.}
  \label{fig:costs}
\end{figure}

With respect to the cost~\eqref{eq:39}, where the rectangle $R$
obviously plays a key role, the most effective strategy consists in a
constant release of parasitoids over the rectangle $R$, corresponding
to the control $q_0^R$ in~\eqref{eq:16}. This solution is somewhat
periodic and displays a maximum, respectively a minimum, of the
running cost at the time $t \approx 33.30$, respectively
$t \approx 30.79$: level plots of the corresponding solutions computed
at these times are in Figure~\ref{fig:qrContour}.

It is evident that the convective term in the first equation
in~\eqref{eq:35} allows the parasitoids to move towards the region
with the highest parasite concentration. On the other hand, the
Laplace operator in the second equation makes the parasites diffuse
everywhere.
\begin{figure}[!ht]
  \includegraphics[width=0.5\linewidth, trim = 57 90 10 40,
  clip=true]{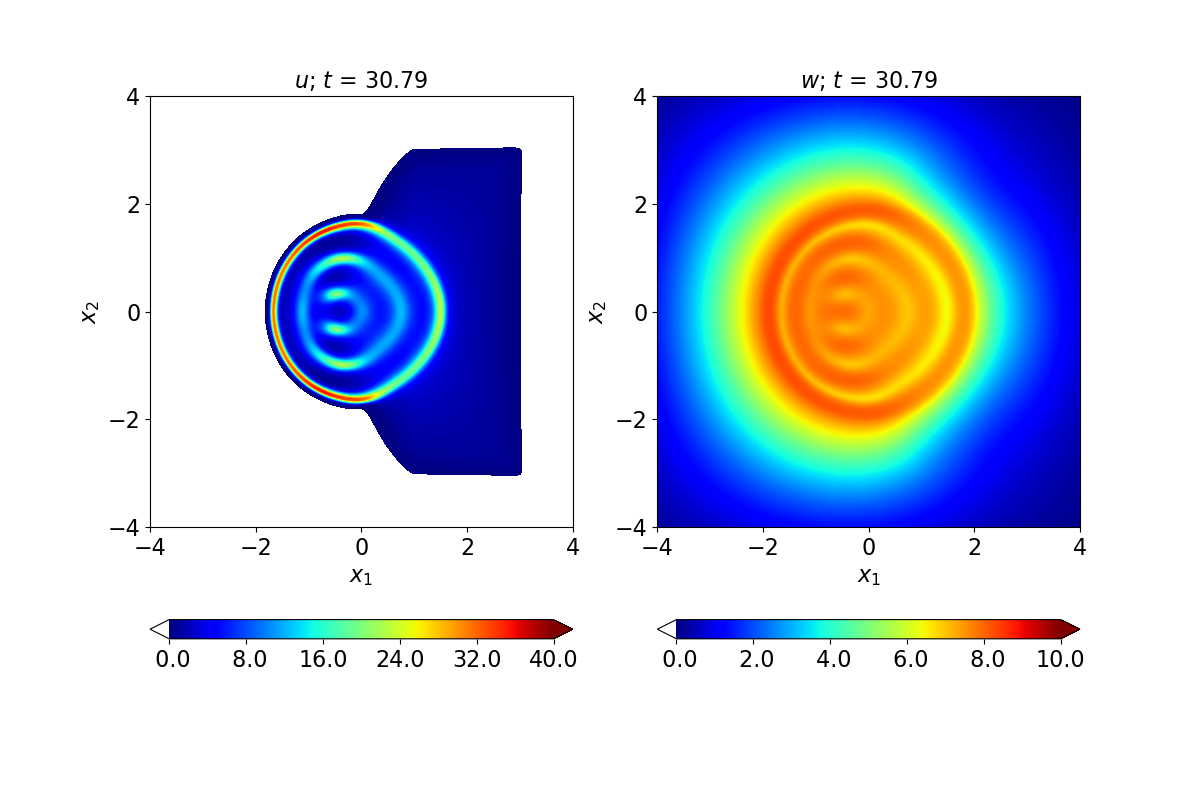}%
  ~%
  \includegraphics[width=0.5\linewidth, trim = 57 90 10 40, clip=true]{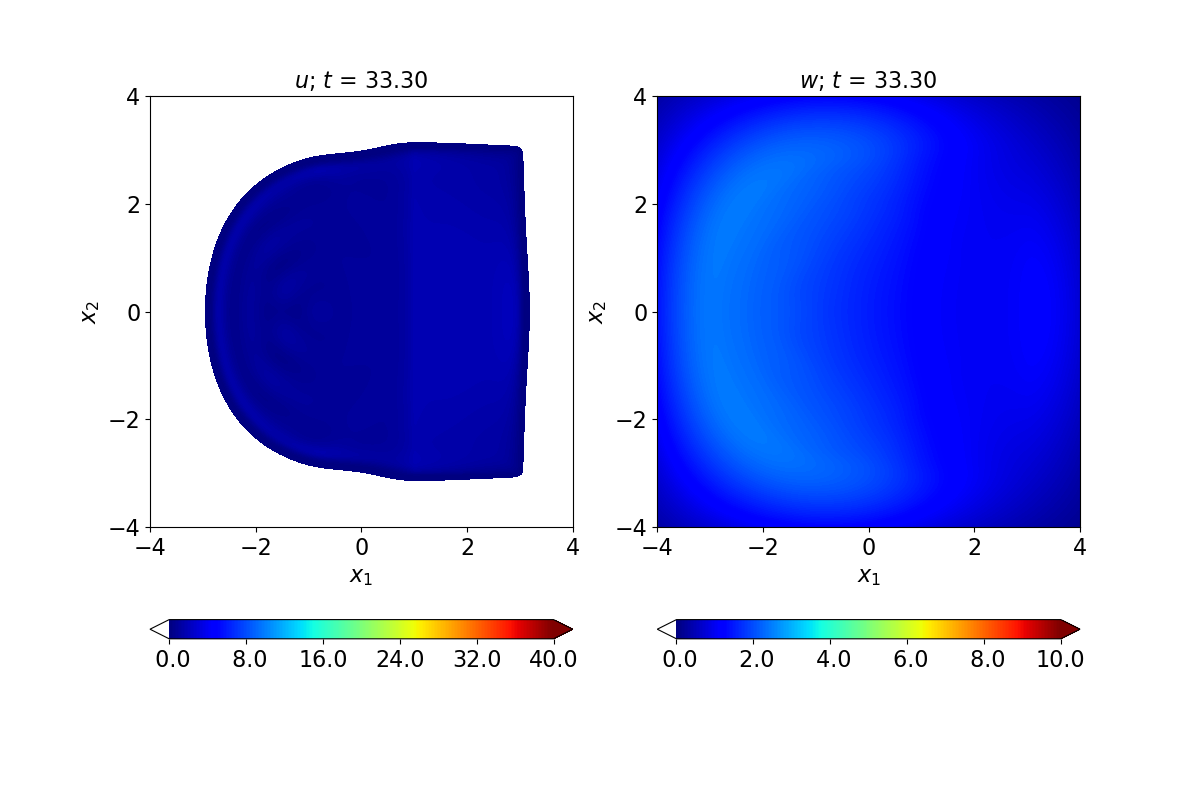}\\
  \caption{Contour plots of the solution to corresponding to the best
    strategy $q_0^R$ in~\eqref{eq:16}. Left, at time $t = 30.79$
    approximately corresponding to a maximum of the running cost and,
    right, at time $t = 33.30$ approximately corresponding to a
    minimum.}
  \label{fig:qrContour}
\end{figure}

\smallskip

We expect that a precise simulation of a real scenario requires a
model more complex than~\eqref{eq:35}--~\eqref{eq:39}, as well as the
obvious tuning of the various parameters. For instance, also
$\alpha, \beta$ and $\delta$ are likely to be better substituted by
\emph{``seasonal''} (i.e., time periodic) functions. While such an
experimental fitting is out of the scopes of the present work, we
remark that the generality of the framework presented here, and in
particular Theorem~\ref{thm:main}, allows to comprehend it.

Boundary conditions deserve a specific treatment on their own. At the
modeling level, the immigration of parasites is neglected in the
present work. At the analytic level, general well posedness and
stability results are currently apparently still missing,
see~\cite{siam2018} for recent preliminary results. The numerical
algorithm to deal with boundary conditions would then be necessarily
adapted.

\section{Analytic Proofs}
\label{sec:AP}


The following lemmas will be of use below. The proofs, where
immediate, are omitted.

\begin{lemma}[{\cite[Formula~(1.8) and Remark~1.16]{Giusti}}]
  \label{lem:approx2}
  Let $\psi \in (\L1 \cap \L\infty \cap \BV) (\reali^n;
  \reali)$. Then, there exists a sequence
  $\psi_h \in \C\infty (\reali^n; \reali)$ such that for
  $h \in \naturali \setminus \{0\}$
  \begin{align}
    \label{eq:6}
    \psi_h \underset{h \to +\infty}{\to}
    & \psi \mbox{ in } \L1 (\reali^n; \reali),
    & \norma{\psi_h}_{\L\infty (\reali^n)} \leq
    & \norma{\psi}_{\L\infty (\reali^n)},
    & \tv (\psi_h) \underset{h \to +\infty}{\to}
    & \tv (\psi).
  \end{align}
\end{lemma}

\begin{lemma}
  \label{lem:approx3}
  Let $\psi \in (\L\infty \cap \BV) (\reali^n; \reali)$. Then, there
  exists a sequence $\psi_h \in \C\infty (\reali^n; \reali)$ such that
  for $h \in \naturali \setminus \{0\}$, $\psi_h \to \psi$\ in
  $\L\infty (\reali^n; \reali)$, so that also $\psi_h \to \psi$ in
  $\Lloc1 (\reali^n; \reali)$, and
  \begin{displaymath}
    \norma{\psi_h}_{\L\infty (\reali^n)}
    \leq
    \norma{\psi}_{\L\infty (\reali^n)}, \quad
    \tv (\psi_h)
    \leq
    \tv (\psi).
  \end{displaymath}
\end{lemma}

\begin{proof}
  Let $\rho$ be a mollifier: $\rho \in \Cc\infty (\reali^n, \reali)$,
  $\rho \geq 0$,
  $\spt \rho \subseteq\left\{x \in \reali^2 \colon \norma{x} \leq
    1\right\}$ and $\int_{\reali^n} \rho = 1$. Define
  $\rho_h (x) = h^n \, \rho (h\,x)$ for
  $h \in \naturali \setminus\{0\}$ and set $\psi_h = \rho_h *
  \psi$. The $\Lloc1$ convergence follows from
  $\norma{\psi_h - \psi}_{\L\infty (\reali^n)} \to 0$, ensured
  by~\cite[Theorem~8.14]{FollandBook}. The $\L\infty$ estimate is a
  consequence of~\cite[Proposition~8.7]{FollandBook}. Finally,
  \cite[Proposition~8.68]{FollandBook} implies the latter bound.
\end{proof}

\begin{lemma}
  \label{lem:approx}
  Let $\psi \in \L\infty (I \times \reali^n; \reali)$ be such that for
  all $t \in I$, $\psi (t) \in \BV (\reali^n; \reali)$. Then, there
  exists a sequence $\psi_h \in \L\infty (I \times \reali^n; \reali)$
  such that for all $h \in \naturali \setminus \{0\}$ and for
  a.e.~$t \in I$,
  $\psi_h (t) \in (\C\infty\cap \BV) (\reali^n; \reali)$,
  $\psi_h (t) \to \psi (t)$ in $\Lloc1 (\reali^n; \reali)$ and
  \begin{displaymath}
    \norma{\psi_h (t)}_{\L\infty (\reali^n)}
    \leq
    \norma{\psi (t)}_{\L\infty (\reali^n)}
    ,\quad
    \tv \!\left(\psi_h (t)\right)
    \leq
    \tv \!\left(\psi (t)\right).
  \end{displaymath}
\end{lemma}

\subsection{About the Parabolic Equation
  \texorpdfstring{$\partial_t w - \mu \, \Delta w = a (t,x) \,w$}{}}
\label{sec:parabolic}

We focus on the parabolic problem:
\begin{equation}
  \label{eq:par}
  \left\{
    \begin{array}{l}
      \partial_t w - \mu \, \Delta w = a (t,x) \,w
      \\
      w(t_o,x) = w_o (x)
    \end{array}
  \right.
  \qquad (t,x) \in I \times \reali^n .
\end{equation}
Similarly to~\cite{parahyp}, solutions to~\eqref{eq:par} are sought as
$\L1$ function defined on $\reali^n$ and all estimates refer to the
$\L1$ or $\L\infty$ norms, see~\eqref{eq:normX}, which is somewhat
unusual in relation to~\eqref{eq:par}.

\begin{definition}
  \label{def:par}
  Let $a \in \L\infty (I \times \reali^n;\reali)$ and
  $w_o \in \L1 (\reali^n;\reali)$. A \emph{solution} to
  problem~\eqref{eq:par} is a function
  $w \in \C0 (I; \L1 (\reali^n; \reali))$ such that
  \begin{equation}
    \label{eq:solNucleo}
    w (t,x) = \left( H_\mu (t) * w_o \right) (x)
    +
    \int_{t_0}^t
    \left( H_\mu (t-\tau) * \left(a (\tau) \, w (\tau)\right) \right)
    (x) \, \d\tau \,.
  \end{equation}
\end{definition}

\begin{lemma}[{\cite[Lemma~2.4]{parahyp}}]
  \label{lem:paraSol}
  Let $a \in \L\infty (I \times \reali^n; \reali)$. Assume that
  $w_o \in \L1 (\reali^n; \reali)$ and
  $w \in \C0 (I; \L1 (\reali^n; \reali))$. Then, the following
  statements are equivalent:
  \begin{enumerate}[leftmargin=*]
  \item The function $w$ solves~\eqref{eq:par} in the sense of
    Definition~\ref{def:par}.
  \item The function $w$ is a weak solution to~\eqref{eq:par}, i.e.,
    for all test functions $\phi \in \Cc2 (I \times \reali^n; \reali)$
    \begin{equation}
      \label{eq:Psol}
      \int_{t_o}^T \int_{\reali^n}
      (w \, \partial_t \phi + \mu \, w \, \Delta \phi + a \, w \, \phi) \,
      \, \d{x} \, \d{t}
      = 0
    \end{equation}
    and $w(t_o,x) = w_o (x)$.
  \end{enumerate}
\end{lemma}

\begin{proposition}[{\cite[Proposition~2.5]{parahyp}}]
  \label{prop:para}
  Fix $a \in \L\infty (I \times \reali^n; \reali)$. Then,
  \eqref{eq:par} generates the process
  \begin{displaymath}
    \begin{array}{ccccccc}
      \mathcal{P}
      & \colon
      & J
      & \times
      & \L1 (\reali^n; \reali)
      & \to
      & \L1 (\reali^n; \reali)
      \\
      &
      &(t_o,t)
      & ,
      & w_o
      & \to
      & w
    \end{array}
  \end{displaymath}
  with $w$ defined as in~\eqref{eq:solNucleo}, with the following
  properties, for a suitable $\mathcal{O} \in \Lloc\infty (I; \reali)$
  that depends only on norms of the map $a$ on $I\times \reali^n$.

  \begin{enumerate}[label=\bf{(P\arabic*)}]
    \setlength{\itemsep}{1pt} \setlength{\parskip}{1pt} \smallskip

  \item \label{it:P1} \textbf{$\boldsymbol{\mathcal{P}}$ is a
      Process:} $\mathcal{P}_{t,t} = \Id$ for all $t \in I$ and
    $\mathcal{P}_{t_2,t_3} \circ \mathcal{P}_{t_1,t_2} =
    \mathcal{P}_{t_1,t_3}$ for all $t_1, t_2, t_3 \in I$, with
    $t_1 \leq t_2 \leq t_3$.  \smallskip

  \item \label{it:P:timereg}\textbf{Regularity in time:} for all
    $w_o \!\in\! \L1 (\reali^n; \reali)$, the map
    $t \!\to\! \mathcal{P}_{t_o,t} w_o$ is in
    $\C0\!  \left(I; \L1 (\reali^n; \reali)\right)$, and, moreover,
    for every $\theta \in \left]0,1\right[$ and for all
    $\tau, \, t_1, \, t_2 \in I$, with $t_2 \geq t_1 \geq \tau > 0$,
    \begin{displaymath}
      \norma{
        \mathcal{P}_{t_o,t_2}w_o - \mathcal{P}_{t_o,t_1}w_o}_{\L1 (\reali^n)}
      \leq
      \norma{w_o}_{\L1 (\reali^n)}\,
      \left[
        \frac{n}{\tau-t_o} + \mathcal{O}(t_2)
      \right] \,
      \modulo{t_2-t_1}^\theta,
    \end{displaymath}

  \item \label{it:P:spacereg}\textbf{Regularity in space:} for all
    $t > t_o$, $w (t) \in \C\infty (\reali^n; \reali)$.

  \item \label{it:P4} \textbf{Regularity in $\boldsymbol{(t,x)}$:} if
    $w_o \in (\L1 \cap \C1) (\reali^n; \reali)$, then
    $(t,x) \to (\mathcal{P}_{t_o,t} w_o) (x) \in \C1 (I \times
    \reali^n; \reali)$.

  \item \label{it:P2} \textbf{$\L1$ continuous dependence on
      $\boldsymbol{w_o}$:} for all $t \in I$, the map
    $\mathcal{P}_{t_o,t} \colon \L1 (\reali^n; \reali) \to \L1
    (\reali^n; \reali)$ is linear and continuous, with
    $\norma{\mathcal{P}_{t_o,t} w_o}_{\L1 (\reali^n)} \leq \mathcal{O}
    (t) \, \norma{w_o}_{\L1 (\reali^n)}$.

  \item \label{it:P6} \textbf{Stability with respect to
      $\boldsymbol{a}$:} let
    $a_1, a_2 \in \L\infty (I \times \reali^n; \reali)$ with
    $a_1-a_2 \in \L1 (I \times \reali^n; \reali)$ and call
    $\mathcal{P}^1, \mathcal{P}^2$ the corresponding processes.  Then,
    for all $t \in I$ and for all
    $w_o \in (\L1\cap \L\infty) (\reali^n; \reali)$,
    \begin{displaymath}
      \norma{
        \mathcal{P}^1_{t_o,t} w_o - \mathcal{P}^2_{t_o,t} w_o}_{\L1 (\reali^n)}
      \leq
      \mathcal{O} (t) \,
      \norma{w_o}_{\L\infty (\reali^n)} \,
      \norma{a_1-a_2}_{\L1 ([t_o,t] \times \reali^n)} .
    \end{displaymath}

  \item \label{it:P3} \textbf{$\L\infty$-estimate:} for all
    $w_o \in (\L1\cap \L\infty) (\reali^n; \reali)$, for all
    $t \in I$,
    $\norma{\mathcal{P}_{t_o,t}w_o}_{\L\infty (\reali^n)} \leq
    \mathcal{O} (t) \, \norma{w_o}_{\L\infty (\reali^n)}$.

  \item\label{item:tvw} \textbf{$\W11$-estimate:} for all
    $w_o \in \L1 (\reali^n; \reali)$, for all $t \in I$ with $t>t_o$,
    \begin{eqnarray*}
      \norma{\nabla (\mathcal{P}_{t_o,t}w_o)}_{\L1 (\reali^n; \reali^n)}
      & \leq
      & \frac{J_n}{\sqrt{\mu \, (t - t_o) }} \, \norma{w_o}_{\L1 (\reali^n)}
      \\
      &
      & \times \left(1 + 2 \, (t-t_o) \, \norma{a}_{\L\infty ([t_o,t] \times \reali^n)}
        e^{\int_{t_o}^t\norma{a (\tau)}_{\L\infty (\reali^n)} \d\tau}\right),
    \end{eqnarray*}
    where $J_n = \frac{\Gamma ((n+1)/2)}{\Gamma (n/2)}$ and $\Gamma$
    is the Gamma function.
  \end{enumerate}
\end{proposition}

\noindent The latter estimate above and~\ref{it:P:spacereg} provide a
$\BV$ bound on the solution $\mathcal{P}_{t_o,t}w_o$ for $t>t_o$.

In the sequel, we need the following strengthened version
of~\ref{item:tvw}.

\begin{proposition}
  \label{prop:PBV}
  Let $a \in \L\infty (I \times \reali^n; \reali)$ and assume
  $w_o \in (\L1 \cap \L\infty \cap \BV) (\reali^n; \reali)$. Call $w$
  the solution to~\eqref{eq:par}. Then, for all $t \in I$,
  $w (t) \in \BV (\reali^n; \reali)$ and the following estimate holds:
  \begin{equation}
    \label{eq:31}
    \tv \left(w (t)\right)
    \leq
    \tv (w_o)
    +
    \frac{2\,J_n}{\sqrt{\mu}} \; \mathcal{O} (t) \;
    \norma{a}_{\L\infty ([t_o,t] \times \reali^n)} \;
    \norma{w_o}_{\L1 (\reali^n)},
  \end{equation}
  where $J_n = \frac{\Gamma ((n+1)/2)}{\Gamma (n/2)}$ and $\Gamma$ is
  the Gamma function.
\end{proposition}

\begin{proof}
  Approximate $w_o$ by means of a sequence $w_o^h$ as defined in
  Lemma~\ref{lem:approx2}. Define $w_h$ through~\eqref{eq:solNucleo}
  by
  \begin{equation}
    \label{eq:30}
    w_h (t,x)
    =
    \left( H_\mu (t) * w^h_o \right) (x)
    +
    \int_{t_0}^t
    \left( H_\mu (t-\tau) * \left(a (\tau) \, w_h (\tau)\right) \right)
    (x) \, \d\tau .
  \end{equation}
  Let $w$ be defined by~\eqref{eq:solNucleo} and compute
  \begin{displaymath}
    \norma{w_h (t) - w (t)}_{\L1 (\reali^n)}
    \leq
    \norma{w_o^h - w_o}_{\L1 (\reali^n)}
    + \int_{t_o}^t\norma{a (\tau)}_{\L\infty (\reali^n)}
    \norma{w_h (\tau) - w (\tau)}_{\L1 (\reali^n)} \d\tau .
  \end{displaymath}
  An application of Gronwall Lemma~\cite[Chapter~I, 1.III]{Walter1970}
  yields
  \begin{displaymath}
    \norma{w_h (t) - w (t)}_{\L1 (\reali^n)}
    \leq
    \norma{w_o^h - w_o}_{\L1 (\reali^n)}
    \exp\left(
      \int_{t_o}^t\norma{a (\tau)}_{\L\infty (\reali^n)} \d\tau
    \right).
  \end{displaymath}
  Thus, as $h$ goes to $+\infty$, $w_h (t)$ converges to $w (t)$ in
  $\L1 (\reali^n; \reali)$ for a.e.~$t\in I$.

  It follows immediately from~\eqref{eq:30} and from the regularity of
  the heat kernel $H_\mu$ that
  $w_h (t) \in \C\infty (\reali^n; \reali)$ for a.e.~$t \in
  I$. Moreover,
  \begin{displaymath}
    \nabla w_h (t,x)
    =
    (H_\mu (t) * \nabla w_o^h) (x)
    +
    \int_{t_o}^t
    \nabla H_\mu (t-\tau) * \left(a (\tau) \, w_h (\tau)\right) (x)
    \d\tau,
  \end{displaymath}
  so that, using the properties of the heat kernel $H_\mu$
  and~\ref{it:P2} in Proposition~\ref{prop:para}, we obtain
  \begin{align*}
    \norma{\nabla w_h (t)}_{\L1 (\reali^n; \reali^n)}
    \leq \
    & \norma{\nabla w_o^h}_{\L1 (\reali^n; \reali^n)}
    \\
    & +
      \int_{t_o}^t
      \norma{\nabla H_\mu (t-\tau)}_{\L1 (\reali^n;\reali^n)}
      \norma{a (\tau)}_{\L\infty (\reali^n)}
      \norma{w_h (\tau)}_{\L1 (\reali^n)} \d\tau
    \\
    \leq \
    & \norma{\nabla w_o^h}_{\L1 (\reali^n; \reali^n)}
      +
      \mathcal{O} (t) \,
      \norma{a}_{\L\infty ([t_o,t] \times \reali^n)} \,
      \norma{w_o}_{\L1 (\reali^n)}
      \int_{t_o}^t
      \frac{J_n}{\sqrt{\mu (t-\tau)}} \d\tau
    \\
    \leq \
    &\norma{\nabla w_o^h}_{\L1 (\reali^n; \reali^n)}
      +
      \mathcal{O} (t) \,
      \norma{a}_{\L\infty ([t_o,t] \times \reali^n)} \,
      \norma{w_o}_{\L1 (\reali^n)} \,
      \frac{2\,J_n}{\sqrt{\mu}} \, \sqrt{t-t_o} \,.
  \end{align*}
  Let now $h \to +\infty$: Lemma~\ref{lem:approx2} and the lower
  semicontinuity of the total variation imply that:
  \begin{align*}
    \tv \left(w (t)\right)
    \leq \
    & \lim_{h \to +\infty} \tv (w_h (t))
      =
      \lim_{h\to+\infty} \norma{\nabla w_h (t)}_{\L1 (\reali^n; \reali^n)}
    \\
    \leq \
    &  \tv (w_o)
      +
      \mathcal{O} (t) \,
      \norma{w_o}_{\L1 (\reali^n)} \,
      \norma{a}_{\L\infty ([t_o,t] \times \reali^n)}
      \frac{2\,J_n}{\sqrt{\mu}} \, \sqrt{t-t_o},
  \end{align*}
  completing the proof.
\end{proof}

We need the following improvements of the estimates in
propositions~\ref{prop:para} and~\ref{prop:PBV} that hold in the case
of positive initial data.

\begin{corollary}
  \label{cor:para}
  Let $a \in \L\infty (I \times \reali^n; \reali)$,
  $w_o \in \L1 (\reali^N; \reali)$ with $w_o \geq 0$. Then,
  \begin{enumerate}[label=\bf{(P\arabic*)},start=9]

    \setlength{\itemsep}{1pt} \setlength{\parskip}{1pt}

  \item \label{it:P:posi} \textbf{Positivity:}
    $\mathcal{P}_{t_o,t} \, w_o \geq 0$ for all $t \in I$.

  \item \label{it:P:priori} \textbf{A priori estimates:} assume that
    $w_o \in (\L1 \cap \L\infty) (\reali^n; \reali)$ and set, for all
    $t \in I$, $A (t) = \sup_{\xi \in \reali^n} a (t,\xi)$. Then,
    \begin{equation}
      \label{eq:12}
      \begin{aligned}
        \norma{\mathcal{P}_{t_o,t} \, w_o}_{\L1 (\reali^n)} \leq \ &
        \norma{w_o}_{\L1 (\reali^n)} \, \exp \int_{t_o}^t A (\tau)
        \d\tau \,,
        \\
        \norma{\mathcal{P}_{t_o,t} \, w_o}_{\L\infty (\reali^n)} \leq
        \ & \norma{w_o}_{\L\infty (\reali^n)} \, \exp \int_{t_o}^t A
        (\tau) \d\tau \,.
      \end{aligned}
    \end{equation}

  \item\label{it:P:stab} \textbf{Stability with respect to
      $\boldsymbol{a}$:} let
    $a_1, a_2 \in \L\infty (I \times \reali^n; \reali)$ with
    $a_1-a_2 \in \L1 (I \times \reali^n; \reali)$ and call
    $\mathcal{P}^1, \mathcal{P}^2$ the corresponding processes.  Then,
    for all $t \in I$ and for all
    $w_o \in (\L1\cap \L\infty) (\reali^n; \reali)$,
    \begin{equation}
      \label{eq:13}
      \begin{aligned}
        & \norma{\mathcal{P}^1_{t_o,t} w_o- \mathcal{P}^2_{t_o,t}
          w_o}_{\L1 (\reali^n)}
        \\
        \leq \ & \norma{w_o}_{\L\infty (\reali^n)} \, e^{\int_{t_o}^t
          \left[\norma{a_1 (\tau)}_{\L\infty (\reali^n)} + \norma{a_2
              (\tau)}_{\L\infty (\reali^n)}\right]\d\tau} \,
        \norma{a_1 - a_2}_{\L1 ([t_o,t]\times\reali^n)} \,.
      \end{aligned}
    \end{equation}

  \item \label{item:tvwPOS} \textbf{$\BV$ estimate:} if
    $w_o \in (\L1 \cap \L\infty \cap \BV) (\reali^n; \reali)$, define
    $A (t) = \sup_{x \in \reali^n} a (t,x)$, then
    \begin{equation}
      \label{eq:32}
      \tv \left(\mathcal{P}_{t_o,t} \, w_o\right)
      \leq
      \tv (w_o)
      +
      \frac{2\,J_n}{\sqrt{\mu}} \, \sqrt{t-t_o} \,
      \norma{a}_{\L\infty ([t_o,t] \times \reali^n)}
      \norma{w_o}_{\L1 (\reali^n)} \,
      e^{\int_{t_o}^t A (\tau) \d\tau},
    \end{equation}
    where $J_n = \frac{\Gamma ((n+1)/2)}{\Gamma (n/2)}$ and $\Gamma$
    is the Gamma function.
  \end{enumerate}
\end{corollary}

\begin{proof}
  The positivity~\ref{it:P:posi} follows from~\cite[Point~6.~in
  Proposition~2.5]{parahyp}, based on~\cite[Chapter~2, Section~4,
  Theorem~9]{FriedmanBook}.

  Starting now from~\eqref{eq:solNucleo}, we have
  \begin{align*}
    w (t,x)
    = \
    & \left( H_\mu (t) * w_o\right) (x)
      +
      \int_{t_o}^t \int_{\reali^n}
      H_\mu (t-\tau, x-\xi) \, a (\tau,\xi) \, w (\tau,\xi) \d{\xi} \d{\tau}
    \\
    \leq \
    & \left( H_\mu (t) * w_o\right) (x)
      +
      \int_{t_o}^t A (\tau) \int_{\reali^n}
      H_\mu (t-\tau, x-\xi) \, w (\tau,\xi) \d{\xi} \d{\tau}  \,.
  \end{align*}
  In both cases of the $\L1$ and $\L\infty$ estimate, an application
  of Gronwall Lemma~\cite[Chapter~I, 1.III]{Walter1970} completes the
  proof of~\ref{it:P:priori}.

  Concerning the stability with respect to $a$, denote
  $w_i(t) = \mathcal{P}_{t_o,t}^i w_o$, for $i=1,2$ and $t \in I$, and
  using~\eqref{eq:solNucleo}, compute
  \begin{align*}
    w_1 (t,x) - w_2 (t,x)
    = \
    &  \int_{t_o}^t\int_{\reali^n}
      H_\mu (t-\tau, x-\xi)
      \left(a_1 (\tau,\xi) \, w_1 (\tau,\xi)
      -
      a_2 (\tau,\xi) \, w_2 (\tau,\xi)\right)
      \d\xi \, \d\tau
    \\
    = \
    & \int_{t_o}^t\int_{\reali^n}
      H_\mu (t-\tau, x-\xi)
      \left(a_1 (\tau,\xi)
      -
      a_2 (\tau,\xi) \right)
      w_1 (\tau,\xi)
      \d\xi \, \d\tau
    \\
    & +
      \int_{t_o}^t\int_{\reali^n}
      H_\mu (t-\tau, x-\xi)
      \, a_2 (\tau,\xi)
      \left( w_1 (\tau,\xi)
      -
      w_2 (\tau,\xi)\right)
      \d\xi \, \d\tau,
  \end{align*}
  so that
  \begin{align*}
    \norma{w_1 (t) - w_2 (t)}_{\L1 (\reali^n)}
    \leq \
    & \int_{t_o}^t \norma{a_1 (\tau) - a_2 (\tau)}_{\L1 (\reali^n)}
      \norma{w_1 (\tau)}_{\L\infty (\reali^n)} \d\tau
    \\
    &
      +
      \int_{t_o}^t \norma{a_2 (\tau)}_{\L\infty (\reali^n)}\,
      \norma{w_1 (\tau) - w_2 (\tau)}_{\L1 (\reali^n)} \d\tau \,.
  \end{align*}
  By Gronwall Lemma,
  \begin{align*}
    &
      \norma{w_1 (t) - w_2 (t)}_{\L1 (\reali^n)}
    \\
    \leq \
    & \int_{t_o}^t \norma{a_1 (\tau) - a_2 (\tau)}_{\L1 (\reali^n)}
      \norma{w_1 (\tau)}_{\L\infty (\reali^n)} \d\tau
      \exp\left(\int_{t_o}^t \norma{a_2 (\tau)}_{\L\infty (\reali^n)} \d\tau\right)
    \\
    \leq \
    & \int_{t_o}^t \norma{a_1 (\tau) - a_2 (\tau)}_{\L1 (\reali^n)}
      \norma{w_o}_{\L\infty (\reali^n)}
      \exp\left(\int_{t_o}^\tau A_1 (s)\d{s}\right)\d\tau
      \exp\left(\int_{t_o}^t \norma{a_2 (\tau)}_{\L\infty (\reali^n)} \d\tau\right)
    \\
    \leq \
    & \norma{w_o}_{\L\infty (\reali^n)} \,
      \exp\left(
      \int_{t_o}^t
      \left(
      \norma{a_1 (\tau)}_{\L\infty (\reali^n)}
      +
      \norma{a_2 (\tau)}_{\L\infty (\reali^n)}
      \right)\d\tau\right) \,
      \norma{a_1 - a_2}_{\L1 ([t_o,t]\times\reali^n)} \,,
  \end{align*}
  completing the proof of~\ref{it:P:stab}.

  Finally, \ref{item:tvwPOS} follows from Proposition~\ref{prop:PBV},
  from~\ref{it:P:posi} and from the $\L\infty$ bound~\eqref{eq:12}.
\end{proof}

\subsection{About the Balance Law
  \texorpdfstring{$\partial_t u +\div (c (t,x) \, u) = b (t,x) \, u +
    q (t,x)$}{}}
\label{subs:Hyp}

We focus on the following Cauchy problem for a linear balance law
\begin{equation}
  \label{eq:hyp}
  \left\{
    \begin{array}{l}
      \partial_t u +\div (c (t,x) \, u) = b (t,x) \, u + q (t,x)
      \\
      u(t_o,x) = u_o (x).
    \end{array}
  \right.
\end{equation}

\noindent Recall the following conditions on the functions defining
problem~\eqref{eq:hyp}:
\begin{enumerate}[label={$\boldsymbol{(b)}$}]
\item \label{b} $b \in \L\infty(I \times \reali^n; \reali)$.
\end{enumerate}
\begin{enumerate}[label={$\boldsymbol{(b+)}$}]
\item \label{b*} $b \in \L\infty(I \times \reali^n; \reali)$
  and $b (t) \in \BV (\reali^n; \reali)$ for $t \in I$.
\end{enumerate}
\begin{enumerate}[label={$\boldsymbol{(c1)}$}]
\item \label{c} The map $c$ satisfies
  $c \in (\C0 \cap \L\infty) (I \times \reali^n; \reali^n)$,
  $c (t) \in \C1(\reali^n; \reali^n)$ for all $t \in I$ and
  $\nabla c \in \L\infty(I \times \reali^n;\reali^{n \times n})$.
\end{enumerate}
\begin{enumerate}[label={$\boldsymbol{(c2)}$}]
\item \label{c*} The map $c$ satisfies
  $c \in (\C0 \cap \L\infty) (I \times \reali^n; \reali^n)$;
  $c (t) \in \C2(\reali^n; \reali^n)$ for all $t \in I$,
  $\nabla c \in \L\infty(I \times \reali^n;\reali^{n \times n})$ and
  $\nabla \div c \in \L1 (I \times \reali^n; \reali^n)$.
\end{enumerate}
\begin{enumerate}[label={$\boldsymbol{(q-)}$}]
\item \label{q}
  $q \in \L\infty(I \times \reali^n; \reali) \cap \L\infty (I; \L1
  (\reali^n; \reali))$.
\end{enumerate}

\begin{definition}
  \label{def:hyp}
  Let~\ref{b}, \ref{c} and~\ref{q} hold and choose
  $u_o \in (\L1 \cap \L\infty) (\reali^n; \reali)$. A \emph{solution}
  to~\eqref{eq:hyp} is a function
  $u \in \C0 (I; \L1 (\reali^n; \reali))$ such that
  \begin{align}
    \label{eq:solCara}
    u (t,x) = \
    & u_o (X (t_o;t,x)) \, \exp\left( \int_{t_o}^t \left(
      b(\tau,X (\tau;t,x)) - \div c \left(\tau,X (\tau;t,x)\right)
      \right) \d\tau \right)
    \\ \nonumber
    & + \int_{t_o}^t q (s, X (s;t,x)) \exp\left(\int_s^t \left(
      b (\tau, X (\tau; t,x)) - \div c
      \left(\tau,X (\tau;t,x)\right) \right)\d \tau\right)\d{s},
  \end{align}
  where
  \begin{equation}
    \label{eq:5}
    t \mapsto X (t;t_o,x_o)
    \quad \mbox{solves the Cauchy Problem} \quad
    \left\{
      \begin{array}{l}
        \dot X = c (t, X)
        \\
        X (t_o) = x_o \,.
      \end{array}
    \right.
  \end{equation}
\end{definition}

\begin{lemma}[{\cite[Lemma~2.7]{parahyp}
    and~\cite[Lemma~5.1]{CHM2011}}]
  \label{lem:hypSol}
  Let~\ref{b}, \ref{c}, \ref{q} hold, Fix
  $u_o \in (\L1 \cap \L\infty) (\reali^n; \reali)$ and
  $u \in \C0 (I; \L1 (\reali^n; \reali))$. Then, the following three
  statements are equivalent:
  \begin{enumerate}[leftmargin=*]
  \item $u$ is a Kru\v zkov solution to~\eqref{eq:hyp},
    i.e.~$u (t_o) = u_o$ and for all $k \in \reali$ and
    $\phi \in \Cc1 (\pint{I} \times \reali^n; \reali_+)$,
    \begin{equation}
      \label{eq:Ksol}
      \int_I \int_{\reali^n}
      \left[
        (u - k) (\partial_t \phi
        +
        c \cdot \nabla \phi)
        +
        (b \,u + q - k \, \div c) \, \phi
      \right]
      \sgn (u - k) \, \d{x} \, \d{t} \geq 0 \,.
    \end{equation}

  \item $u$ is a weak solution to~\eqref{eq:hyp}, i.e.~$u (t_o) = u_o$
    and for all $\phi \in \Cc1 (\pint{I} \times \reali^n;\reali)$,
    \begin{equation}
      \label{eq:Hsol}
      \int_I \int_{\reali^n}
      (u \, \partial_t \phi + u \, c \cdot \nabla \phi + (b \, u +q) \phi) \,
      \d{x} \, \d{t} = 0 \,.
    \end{equation}

  \item $u$ solves~\eqref{eq:hyp} in the sense of
    Definition~\ref{def:hyp}.
  \end{enumerate}
\end{lemma}

\noindent The proof amounts to mix the techniques used
in~\cite[Lemma~2.7]{parahyp} and~\cite[Lemma~5.1]{CHM2011}.

We recall a different approach to the study of linear balance laws of
type~\eqref{eq:hyp}, which is adopted
in~\cite[Lemma~3.4]{KPS2018}. That Lemma guarantees the existence of a
weak solution, in the sense of~\eqref{eq:Hsol} in
Lemma~\ref{lem:hypSol}, and provides an explicit formula for the
solution in terms of characteristics, corresponding exactly
to~\eqref{eq:solCara}. The regularity requirements in~\cite{KPS2018}
on the functions defining problem~\eqref{eq:hyp} are the following:
for $T\in\reali$, $T>0$,
\begin{align*}
  u_o \in \
  & \L1 (\reali^n;\reali),
  &
    b \in \
  & \L1 ((0,T);\L\infty (\reali^n; \reali)),
  &
    q \in \
  & \L1 ((0,T);\L1 (\reali^n;\reali)),
\end{align*}
and $c \in \C0 ((0,T);\C1 (\reali^n; \reali^n))$ is globally Lipschitz
continuous in space.  Notice that, for $T \in \reali$, $T>0$, our
assumptions~\ref{b}, \ref{c} and~\ref{q} are stronger than those
required in~\cite[Lemma~3.4]{KPS2018}, allowing to apply that result
in the present setting.

The next proposition is not only an extension
of~\cite[Proposition~2.8]{parahyp} to the present setting, but it also
improves it sharply.

\begin{proposition}
  \label{prop:hyper}
  Under the assumptions~\ref{b}, \ref{c} and~\ref{q}, the Cauchy
  Problem~\eqref{eq:hyp} generates the map
  \begin{displaymath}
    \begin{array}{ccccccc}
      \mathcal{H}
      & \colon
      & J
      & \times
      & \mathcal{U}
      & \to
      & \mathcal{U}
      \\
      &
      & (t_o,t)
      & ,
      & u_o
      & \to
      & u
    \end{array}
  \end{displaymath}
  where $u$ is defined by~\eqref{eq:solCara}, with the following
  properties:
  \begin{enumerate}[label=\bf{(H\arabic*)}]
    \setlength{\itemsep}{1pt} \setlength{\parskip}{1pt} \smallskip

  \item\label{item:uProcess} \textbf{$\boldsymbol{\mathcal{H}}$ is a
      process:} $\mathcal{H}_{t,t} = \Id$ for all $t \in I$ and
    $\mathcal{H}_{t_2,t_3} \circ \mathcal{H}_{t_1,t_2} =
    \mathcal{H}_{t_1,t_3}$ for all $t_1, t_2, t_3 \in I$, with
    $t_1 \leq t_2 \leq t_3$.  \smallskip

  \item\label{item:upos} \textbf{Positivity:} if $q \geq 0$ and
    $u_o \in \mathcal{U}^+$, then
    $\mathcal{H}_{t_o, t} \, u_o \in \mathcal{U}^+$ for all $t \in I$.

  \item\label{item:uL1} \textbf{$\L1$ continuous dependence on
      $\boldsymbol{u_o}$:} for all $t \in I$ the map
    $\mathcal{H}_{t_o,t} \colon \mathcal{U} \to \mathcal{U}$ is
    linear, continuous and
    \begin{displaymath}
      \norma{\mathcal{H}_{t_o,t} u_o}_{\L1 (\reali^n)}
      \leq
      \left(
        \norma{u_o}_{\L1 (\reali^n)}
        +
        \norma{q}_{\L1 ([t_o,t]\times\reali^n)}
      \right)
      \exp \int_{t_o}^t \norma{b(\tau)}_{\L\infty (\reali^n)} \d\tau \,.
    \end{displaymath}
    Moreover, if $u_o \geq 0$ and $q \geq 0$, then
    \begin{displaymath}
      \norma{\mathcal{H}_{t_o,t} u_o}_{\L1 (\reali^n)}
      \leq
      \left(
        \norma{u_o}_{\L1 (\reali^n)}
        +
        \norma{q}_{\L1 ([t_o,t]\times\reali^n)}
      \right)
      \exp \int_{t_o}^t \left(\sup_{x\in \reali^n} b (\tau,x)\right) \d\tau \,.
    \end{displaymath}

  \item\label{item:uLinf} \textbf{$\L\infty$--estimate:} for all
    $u_o \in \mathcal{U}$, for all $t \in I$,
    \begin{eqnarray*}
      \norma{\mathcal{H}_{t_o,t} u_o}_{\L\infty (\reali^n)}
      & \leq
      & \left(
        \norma{u_o}_{\L\infty (\reali^n)}
        +
        \norma{q}_{\L1 ([t_o,t]; \L\infty (\reali^n))}
        \right)
      \\
      &
      & \times
        \exp \int_{t_o}^t
        \left(
        \norma{b(\tau)}_{\L\infty (\reali^n)}
        +
        \norma{\div c (\tau)}_{\L\infty (\reali^n)}
        \right) \d\tau .
    \end{eqnarray*}
    Moreover, if $u_o \geq 0$ and $q \geq 0$, then
    \begin{eqnarray*}
      \norma{\mathcal{H}_{t_o,t} u_o}_{\L\infty (\reali^n)}
      & \leq
      & \left(
        \norma{u_o}_{\L\infty (\reali^n)}
        +
        \norma{q}_{\L1 ([t_o,t]; \L\infty (\reali^n))}
        \right)
      \\
      &
      & \times
        \exp \int_{t_o}^t
        \left(
        \left(\sup_{x\in \reali^n} b (\tau,x)\right)
        +
        \norma{\div c (\tau)}_{\L\infty (\reali^n)}
        \right) \d\tau .
    \end{eqnarray*}

  \item\label{item:ustab} \textbf{Stability with respect to
      $\boldsymbol{b,c,q}$:} if $b, \tilde{b}$ satisfy~\ref{b*} with
    $b -\tilde{b} \in \L1 (I \times \reali^n; \reali)$; $c, \tilde{c}$
    satisfy~\ref{c*} with
    $\div (c -\tilde{c}) \in \L1 (I \times \reali^n; \reali)$ and
    $q, \tilde q$ satisfy~\ref{q*}. Call
    $\mathcal{H}, \tilde{\mathcal{H}}$ the corresponding
    processes. Then, for all $t \in I$ and for all
    $u_o \in \mathcal{U}$,
    \begin{align*}
      & \norma{
        \mathcal{H}_{t_o,t} u_o
        -
        \tilde{\mathcal{H}}_{t_o,t} u_o}_{\L1 (\reali^n)}
      \\
      \leq \
      & \mathcal{O}_1 (t) \,
        \norma{c-\tilde{c}}_{\L1 ([t_o,t];\L\infty (\reali^n; \reali^n))}
        \biggl[ \norma{u_o}_{\L\infty (\reali^n)} + \tv (u_o)
      \\
      & \quad \left.
        +  \int_{t_o}^t \!\!
        \left(
        \max
        \left\{
        \norma{q(\tau)}_{\L\infty (\reali^n)},
        \norma{\tilde q(\tau)}_{\L\infty (\reali^n)}
        \right\}
        + \max\left\{\tv\left( q(\tau)\right), \,
        \tv\left(\tilde q (\tau)\right)\right\}\right) \d\tau
        \right]
      \\
      & +
        \mathcal{O}_2 (t)  \,
        \norma{q- \tilde q}_{\L1 ([t_o,t] \times \reali^n)}
      \\
      & +
        \mathcal{O}_2 (t) \left(
        \norma{u_o}_{\L\infty (\reali^n)}
        + \int_{t_o}^t \max
        \left\{
        \norma{q(\tau)}_{\L\infty (\reali^n)},
        \norma{\tilde q(\tau)}_{\L\infty (\reali^n)}
        \right\} \d\tau
        \right)
      \\
      & \quad \times
        \left(
        \norma{b-\tilde{b}}_{\L1 ([t_o,t]\times\reali^n)}
        +
        \norma{\div (c-\tilde{c})}_{\L1 ([t_o,t]\times\reali^n)}
        \right),
    \end{align*}
    where
    \begin{align*}
      \mathcal{O}_1 (t)
      = \
      & \exp
        \int_{t_o}^t
        \max
        \left\{
        \norma{b(\tau)}_{\L\infty (\reali^n)},
        \norma{\tilde b(\tau)}_{\L\infty (\reali^n)}
        \right\}
        \d\tau
      \\
      & \times \exp
        \int_{t_o}^t
        \max\left\{
        \norma{\nabla c (\tau)}_{\L\infty (\reali^n; \reali^{n \times n})},
        \norma{\nabla \tilde{c} (\tau)}_{\L\infty (\reali^n; \reali^{n \times n})}
        \right\}
        \d\tau
      \\
      & \times
        \left[
        1
        +
        \int_{t_o}^t
        \max\left\{
        \begin{array}{c}
          \tv\left(b (s)\right)
          +
          \norma{\nabla\div c (s)}_{\L1 (\reali^n; \reali^n)},
          \\
          \tv \left(\tilde{b} (s)\right)
          + \norma{\nabla \div \tilde{c} (s) }_{\L1 (\reali^n;\reali^n)}
        \end{array}
      \right\}\d{s}
      \right],
      \\
      \mathcal{O}_2 (t)
      = \
      & \exp
        \int_{t_o}^t
        \max
        \left\{
        \norma{b(\tau)}_{\L\infty (\reali^n)},
        \norma{\tilde b(\tau)}_{\L\infty (\reali^n)}
        \right\}
        \d\tau.
    \end{align*}

  \item\label{item:utv} \textbf{Total variation bound:} let~\ref{b*},
    \ref{c*} and~\ref{q*} hold. If $u_o \in \mathcal{U}$, then, for
    all $t \in I$,
    \begin{align*}
      \tv\left(\mathcal{H}_{t_o,t} u_o \right)
      \leq \
      & \mathcal{O} (t)
        \left(\norma{u_o}_{\L\infty (\reali^n)} + \tv (u_o)
        + \int_{t_o}^t
        \left(
        \norma{q(\tau)}_{\L\infty (\reali^n)}
        +
        \tv\left(q (\tau)\right) \right)\d\tau
        \right) ,
    \end{align*}
    where
    \begin{align*}
      \mathcal{O} (t)
      = \
      & \exp \left(
        \int_{t_o}^t \left(
        \norma{b(\tau)}_{\L\infty (\reali^n)}
        +
        \norma{\nabla c (\tau)}_{\L\infty (\reali^n; \reali^{n \times n})}
        \right)\d\tau
        \right)
      \\
      & \quad \times \left(
        1
        +
        \int_{t_o}^t \left(
        \tv\left(b (\tau)\right)
        +
        \norma{\nabla \div c (\tau)}_{\L1 (\reali^n; \reali^n)}
        \right) \d\tau
        \right).
    \end{align*}

  \item\label{item:uLipt} \textbf{Regularity in time:} let~\ref{b*},
    \ref{c*} and~\ref{q*} hold. For all $u_o \in \mathcal{U}$, the map
    $t \to \mathcal{H}_{t_o,t} u_o$ is in
    $\C{0,1} \left(I; \L1(\reali^n; \reali)\right)$, moreover for all
    $t_1, t_2 \in I$, with $\mathcal{O} (t)$ as above,
    \begin{align*}
      & \norma{
        \mathcal{H}_{t_o,t_2} u_o - \mathcal{H}_{t_o,t_1} u_o}_{\L1 (\reali^n)}
      \\
      \leq \
      & \mathcal{O} (t_1 \vee t_2) \!
        \left(
        \norma{u_o}_{\L\infty (\reali^n)}
        +
        \tv (u_o)\!
        + \!\int_{t_o}^{t_1 \vee t_2}\!\!
        \left(
        \norma{q(\tau)}_{\L\infty (\reali^n)}
        +
        \tv\left(q (\tau)\right) \right)\d\tau
        \right)
        \modulo{t_2-t_1} .
    \end{align*}

  \item\label{item:uSpt} \textbf{Finite propagation speed:} if, for
    all $t \in I$, the map $x \to q (t,x)$ is compactly supported and
    $u_o \in \mathcal{U}$ has compact support, then, for $t\in I$
    also, $\spt \mathcal{H}_{t_o,t}u_o$ is compact.
  \end{enumerate}
\end{proposition}

\begin{proof}
  Statement~\ref{item:uProcess} directly follows from
  Definition~\ref{def:hyp}, Lemma~\ref{lem:hypSol}
  and~\cite[Lemma~3.4]{KPS2018}, thanks to~\ref{b}, \ref{c}
  and~\ref{q}.  Using~\eqref{eq:solCara}, points~\ref{item:upos},
  \ref{item:uLinf} and~\ref{item:uSpt} are ensured.

  To get the $\L1$ bound~\ref{item:uL1}, exploit the change of
  variable $y = X (s;t,x)$, see also~\cite[\S~5.1]{CHM2011}. Denoting
  the Jacobian of this change of variable by
  $J (t,y) = \det\left(\nabla_x X (t;s,y)\right)$, $J$ solves
  \begin{displaymath}
    \frac{\d{J (t,y)}}{\d{t}} = \div c (t, X (t;s,y)) \, J (t,y)
    \qquad \mbox{ with }
    \qquad
    J (s,y) = 1.
  \end{displaymath}
  Thus,
  $J (t,y) = \exp\left(\int_s^t \div c (\tau, X (\tau;s,y))
    \d\tau\right)$, so that $J (t,y) > 0$ for $t \in I$
  and~\ref{item:uL1} follows.

  To prove the remaining points, we exploit the techniques used in the
  proof of~\cite[Lemma~4.4 and Lemma~4.6]{siam2018} for an initial
  boundary value problem for a conservation law, thus without source
  term. To this aim, we approximate $b$, respectively $q$, by a
  sequence $b_h$, respectively $q_h$, as in
  Lemma~\ref{lem:approx}. Regularize also the initial datum $u_o$ and
  call $u_o^h \in \C\infty (\reali^n; \reali)$ the sequence defined by
  Lemma~\ref{lem:approx2}. Using~\eqref{eq:solCara}, define the
  corresponding sequence $u_h$ of solutions to
  \begin{displaymath}
    \left\{
      \begin{array}{l}
        \partial_t u_h +\div (c (t,x) \, u_h) = b_h (t,x) \, u_h + q_h (t,x)
        \\
        u_h(t_o,x) = u_o^h (x) \,,
      \end{array}
    \right.
  \end{displaymath}
  so that
  \begin{align}
    \label{eq:17}
    u_h (t,x) = \
    & u_o^h (X (t_o;t,x)) \, \exp\left( \int_{t_o}^t \left(
      b_h(\tau,X (\tau;t,x)) - \div c \left(\tau,X (\tau;t,x)\right)
      \right) \d\tau \right)
    \\ \nonumber
    & + \int_{t_o}^t q_h (s, X (s;t,x)) \exp\left(\int_s^t \left(
      b_h (\tau, X (\tau; t,x)) - \div c
      \left(\tau,X (\tau;t,x)\right) \right)\d \tau\right)\d{s},
  \end{align}
  where $X$ is defined in~\eqref{eq:5}. Observe that for
  a.e.~$t \in I$, the map $x \to u_h (t,x)$ is of class $\C1$, due to
  Lemma~\ref{lem:approx}, applied to both $b$ and $q$, and
  to~\ref{c*}.

  Pass now to~\ref{item:utv}. Differentiate the solution
  to~\eqref{eq:5} with respect to the initial point, that is, for
  $\tau \in [t_o,t]$,
  \begin{align*}
    \nabla_x X (\tau;t,x) = \
    & \Id + \int_t^\tau \nabla_x c (s, X (s;t,x)) \, \nabla_x X (s;t,x)\d{s},
    \\
    \norma{\nabla_x X (\tau;t,x)} \leq \
    & 1 + \int_\tau^t \norma{\nabla_x c (s, X (s;t,x))} \; \norma{\nabla_x X (s;t,x)} \d{s},
  \end{align*}
  so that, by Gronwall Lemma,
  \begin{equation}
    \label{eq:7}
    \norma{\nabla_x X (\tau;t,x)} \leq
    \exp \left(
      \int_\tau^t
      \norma{\nabla_x c (s)}_{\L\infty (\reali^n;\reali^{n \times n})}
    \right) \d{s}.
  \end{equation}
  By~\eqref{eq:17} and the properties of $u_o^h$, the gradient
  $\nabla u_h (t)$ is well defined and continuous:
  \begin{align*}
    \nabla u_h (t,x)
    = \
    & \exp \left(
      \int_{t_o}^t \left(b_h - \div c \right) (\tau, X (\tau;t,x)) \d\tau
      \right)
      \biggl(
      \nabla u_o^h(X (t_o;t,x)) \, \nabla_x X (t_o;t,x)
    \\
    & \left. + \, u_o^h (X (t_o;t,x))
      \int_{t_o}^t \nabla \left(b_h - \div c\right)\! (\tau,X (\tau;t,x)) \,
      \nabla_xX (\tau;t,x) \d\tau
      \right)
    \\
    & + \int_{t_o}^t \exp \left(
      \int_s^t \left(b_h - \div c \right) (\tau, X (\tau;t,x)) \d\tau
      \right)
      \Bigl( \nabla q_h (s, X (s;t,x)) \, \nabla_xX (s;t,x)
    \\
    & \qquad \left. + q_h (s, X (s;t,x))
      \int_s^t  \nabla \left(b_h - \div c\right)\! (\tau,X (\tau;t,x)) \,
      \nabla_xX (\tau;t,x) \d\tau\right) \d{s}.
  \end{align*}
  Therefore, for every $t \in I$, we use the change of variable
  described at the beginning of the proof together with~\eqref{eq:7}
  to get
  \begin{align}
    \nonumber
    &\norma{\nabla u_h (t)}_{\L1 (\reali^n; \reali^n)}
    \\ \nonumber
    \leq \
    & \exp \left(\int_{t_o}^t \norma{b_h(\tau)}_{\L\infty (\reali^n)} \d\tau \right) \exp
      \left(\int_{t_o}^t
      \norma{\nabla c (\tau)}_{\L\infty (\reali^n; \reali^{n \times n})}\d\tau\right)
    \\  \label{eq:8}
    & \times \Biggl[
      \norma{\nabla u_o^h}_{\L1 (\reali^n)}
      + \int_{t_o}^t \norma{\nabla q_h (\tau)}_{\L1 (\reali^n; \reali^n)} \d\tau
    \\ \nonumber
    &\quad \left.
      +\left(
      \norma{u_o^h}_{\L\infty (\reali^n)}
      + \int_{t_o}^t \norma{q_h(\tau)}_{\L\infty (\reali^n)} \d{\tau}\right)
      \int_{t_o}^t \norma{\nabla
      (b_h - \div c) (\tau)}_{\L1 (\reali^n; \reali^n)} \d\tau
      \right] .
  \end{align}

  Let $u$ be defined as in~\eqref{eq:solCara}: Lemma~\ref{lem:approx}
  and Lemma~\ref{lem:approx2} imply that $u_h \to u$ in
  $\L1 (\reali^n; \reali)$. By the lower semicontinuity of the total
  variation, by~\eqref{eq:8} and~\eqref{eq:6}, for $t\in I$ we obtain
  \begin{align}
    \label{eq:14}
    &  \tv (u (t)) \leq
      \lim_h \tv (u_h (t)) =
      \lim_h \norma{\nabla u_h (t)}_{\L1(\reali^n; \reali^n)}
    \\
    \nonumber
    \leq  \
    &  \exp \left(
      \int_{t_o}^t \left( \norma{b(\tau)}_{\L\infty (\reali^n)}
      +
      \norma{\nabla c (\tau)}_{\L\infty (\reali^n; \reali^{n \times n})}
      \right)\d\tau
      \right)
      \Biggl[
      \tv (u_o) + \int_{t_o}^t \tv \left(q (\tau)\right)\d\tau
    \\ \nonumber
    & \qquad\left. +
      \left( \norma{u_o}_{\L\infty (\reali^n)}
      + \int_{t_o}^t \norma{q(\tau)}_{\L\infty (\reali^n)} \d\tau\right)
      \int_{t_o}^t
      \left( \tv \left(b (\tau)\right)
      +
      \norma{\nabla \div c (\tau)}_{\L1 (\reali^n; \reali^n)}\right) \d\tau
      \right],
  \end{align}
  concluding the proof of~\ref{item:utv}.

  The proof of~\ref{item:uLipt}, is entirely analogous, leading to
  \begin{displaymath}
    \norma{u (t_2) - u (t_1)}_{\L1 (\reali^n)} \leq
    \tv \left(u\left(\max\{t_1,t_2\}\right)\right) \modulo{t_2-t_1}.
  \end{displaymath}

  To prove~\ref{item:ustab}, we follow the idea of the proof
  of~\cite[Lemma~4.6]{siam2018}, adapting it to the present
  setting. With obvious notation, we denote by $b_h$ and $\tilde b_h$
  sequences of functions converging to $b$ and $\tilde b$, with the
  properties in Lemma~\ref{lem:approx}. Similarly, we denote by $q_h$
  and $\tilde q_h$ sequences of functions converging to $q$ and
  $\tilde q$, with the properties in Lemma~\ref{lem:approx}. Consider
  also the regularization of the initial datum
  $u_o^h \in \C\infty (\reali^n; \reali)$ provided by
  Lemma~\ref{lem:approx2}. 
  For $\theta \in [0,1]$ set
  \begin{align*}
    b^\theta_h (t,x) = \
    & \theta \, b_h (t,x) + (1 - \theta) \, \tilde{b}_h (t,x),
    & c^\theta (t,x) = \
    & \theta \, c (t,x) + (1 - \theta) \, \tilde{c} (t,x),
    \\
    q^\theta_h (t,x) = \
    & \theta \, q_h (t,x) + (1 - \theta) \, \tilde{q}_h (t,x).
  \end{align*}
  Let $u^\theta_h$ be the solution to
  \begin{displaymath}
    \nonumber
    \left\{
      \begin{array}{l}
        \partial_t u^\theta_h
        +
        \div \left(c^\theta (t,x) \, u^\theta_h\right)
        =
        b^\theta_h (t,x) \, u^\theta_h + q_h^\theta (t,x)
        \\
        u^\theta_h (t_o,x)= u_o^h (x) \,,
      \end{array}
    \right.
    \quad \mbox{ where } \quad
    \left\{
      \begin{array}{l}
        \dot X^\theta = c^\theta (t, X^\theta)
        \\
        X^\theta (t_o) = x_o \,,
      \end{array}
    \right.
  \end{displaymath}
  that is
  \begin{equation}
    \label{eq:9}
    \begin{aligned}
      u^\theta_h (t,x) = \ & u_o^h (X^\theta (t_o;t,x)) \, \exp\left(
        \int_{t_o}^t \left( b^\theta_h - \div c^\theta\right)
        \left(\tau,X^\theta (\tau;t,x)\right) \d\tau \right)
      \\
      & + \int_{t_o}^t q^\theta_h \left(s,X^\theta (s;t,x)\right)
      \exp\left( \int_s^t \left( b^\theta_h - \div c^\theta\right)
        \left(\tau,X^\theta (\tau;t,x)\right) \d\tau \right) \d{s}.
    \end{aligned}
  \end{equation}
  Compute the derivative of $X^\theta$ with respect to $\theta$,
  recalling that $X^\theta (t;t,x)=x$ for all $\theta$:
  \begin{displaymath}
    \left\{
      \begin{array}{l@{}}
        \partial_t \partial_\theta X^\theta (\tau;t,x) =
        c (\tau,X^\theta (\tau;t,x)) - \tilde{c}(\tau,X^\theta (\tau;t,x))
        +
        \nabla c^\theta (\tau,X^\theta (\tau;t,x)) \,
        \partial_\theta X^\theta (\tau;t,x)
        \\
        \partial_\theta X^\theta (t;t,x)=0.
      \end{array}
    \right.
  \end{displaymath}
  The solution to the above problem satisfies
  \begin{align}
    \nonumber
    \partial_\theta X^\theta (\tau;t,x) = \
    & \int_t^\tau \exp\left(
      \int_s^\tau \nabla c^\theta (\sigma,  X^\theta (\sigma;t,x)) \d\sigma
      \right)
      \left(c-\tilde{c}\right) (s,  X^\theta (s;t,x)) \d{s}
    \\
    \label{eq:10}
    = \
    &  \int_\tau^t \exp\left(
      \int_\tau^s - \nabla c^\theta (\sigma,  X^\theta (\sigma;t,x)) \d\sigma
      \right)
      \left(\tilde{c} - c\right) (s,  X^\theta (s;t,x)) \d{s}.
  \end{align}
  Derive~\eqref{eq:9} with respect to $\theta$:
  \begin{align*}
    & \partial_\theta u^\theta_h (t,x)
    \\
    = \
    & \exp\left(
      \int_{t_o}^t (b^\theta_h - \div c^\theta)
      (\tau, X^\theta (\tau;t,x))\d\tau\right)
      \biggl\{
      \nabla u_o^h  (X^\theta (t_o;t,x)) \; \partial_\theta X^\theta (t_o; t,x)
    \\
    & \quad
      + u_o^h (X^\theta (t_o;t,x)) \int_{t_o}^t \left(
      b_h - \tilde{b}_h - \div (c - \tilde{c})
      \right)(\tau, X^\theta (\tau;t,x))\d\tau
    \\
    & \quad \left.
      + u_o^h (X^\theta (t_o;t,x)) \int_{t_o}^t
      \nabla (b^\theta_h - \div c^\theta)(\tau, X^\theta (\tau;t,x))\;
      \partial_\theta X^\theta (\tau;t,x)
      \d\tau
      \right\}
    \\
    & +\int_{t_o}^t
      \exp\left(
      \int_s^t (b^\theta_h - \div c^\theta)
      (\tau, X^\theta (\tau;t,x))\d\tau\right)
    \\
    & \times \biggl\{(q_h - \tilde q_h) (x, X^\theta (s;t,x))
      + \nabla q_h^\theta (s, X^\theta (s;t,x))
      \; \partial_\theta X^\theta (s;t,x)
    \\
    & \qquad + q_h^\theta (s,X^\theta (s;t,x))
      \int_s^t \left(
      b_h - \tilde{b}_h - \div (c - \tilde{c})
      \right)(\tau, X^\theta (\tau;t,x))\d\tau
    \\
    & \qquad \left. + q_h^\theta (s,X^\theta (s;t,x))
      \int_s^t  \nabla (b^\theta_h - \div c^\theta)(\tau, X^\theta (\tau;t,x))
      \;  \partial_\theta X^\theta (\tau;t,x)
      \d\tau
      \right\} \d{s}
    \\
    \leq \
    & \exp\left(\int_{t_o}^t
      (b^\theta_h - \div c^\theta) (\tau, X^\theta (\tau;t,x))\d\tau\right)
      \Biggl\{
      \int_{t_o}^t (q_h - \tilde q_h) (x, X^\theta (s;t,x)) \d{s}
    \\
    & + \left(\nabla u_o^h  (X^\theta (t_o;t,x))
      + \int_{t_o}^t  \nabla q_h^\theta (\tau, X^\theta (\tau;t,x)) \d\tau
      \right)
    \\
    & \quad \times \int_{t_o}^t \exp\left(
      \int_{t_o}^s - \nabla c^\theta (\sigma,  X^\theta (\sigma;t,x)) \d\sigma
      \right)
      \left(\tilde{c} - c\right) (s,  X^\theta (s;t,x)) \d{s}
    \\
    & + \left( u_o^h (X^\theta (t_o;t,x))
      + \int_{t_o}^t  q_h^\theta (\tau,X^\theta (\tau;t,x)) \d\tau
      \right)
    \\
    & \times \Biggl[
      \int_{t_o}^t \left(
      b_h - \tilde{b}_h - \div (c - \tilde{c})
      \right)\!(\tau, X^\theta (\tau;t,x))\d\tau
    \\
    & \qquad +
      \int_{t_o}^t
      \nabla (b^\theta_h - \div c^\theta)(\tau, X^\theta (\tau;t,x))
    \\
    &\qquad \left.\left. \times \left[
      \int_\tau^t \exp\left(
      \int_\tau^s - \nabla c^\theta (\sigma,  X^\theta (\sigma;t,x)) \d\sigma
      \right)
      \left(\tilde{c} - c\right) (s,  X^\theta (s;t,x)) \d{s}
      \right] \d\tau\right]
      \right\},
  \end{align*}
  where we made use of~\eqref{eq:10}. Call $u_h$ and $\tilde{u}_h$ the
  functions defined by~\eqref{eq:9} for $\theta=0$ and $\theta=1$,
  that is $u_h =u_h^{\theta=0}$ and $\tilde{u}_h =u_h^{\theta=1}$.
  Compute
  \begin{equation}
    \label{eq:11}
    \norma{u_h (t) - \tilde{u}_h (t)}_{\L1 (\reali^n)}
    \leq
    \int_{\reali^n}
    \modulo{\int_0^1 \partial_\theta u^\theta_h (t,x) \d\theta}\d{x}
    \leq
    \int_0^1 \int_{\reali^n}
    \modulo{\partial_\theta u^\theta_h (t,x) } \d{x} \d\theta.
  \end{equation}
  Exploiting the change of variable introduced at the beginning of the
  proof, compute
  \begin{align*}
    & \int_{\reali^n} \modulo{\partial_\theta u^\theta_h (t,x) } \d{x}
    \\
    \leq \
    & \exp\left(
      \int_{t_o}^t \norma{b^\theta_h(\tau)}_{\L\infty (\reali^n)} \d\tau
      \right)
      \Biggl\{
      \int_{t_o}^t \norma{(q_h - \tilde q_h) (\tau)}_{\L1 (\reali^n)} \d\tau
    \\
    & + \left(\int_{\reali^n} \modulo{\nabla u_o^h (y)}\d{y}
      +
      \int_{t_o}^t \int_{\reali^n}
      \modulo{\nabla q_h^\theta (\tau,y)} \d{y}\d\tau
      \right)
    \\
    & \quad \times
      \exp\left(
      \int_{t_o}^t
      \norma{\nabla c^\theta (\sigma)}_{\L\infty (\reali^n; \reali^{n \times n})}
      \d\sigma
      \right)
      \int_{t_o}^t\norma{ (c-\tilde{c}) (s)}_{\L\infty (\reali^n;\reali^n)} \d{s}
    \\
    & \quad + \left( \norma{u_o^h}_{\L\infty (\reali^n)}
      + \int_{t_o}^t \norma{q^\theta_h (\tau)}_{\L\infty (\reali^n)} (\tau)\d\tau \right)
      \int_{t_o}^t
      \norma{(b_h - \tilde{b}_h
      - \div (c - \tilde{c})) (\tau)}_{\L1 (\reali^n)}\d\tau
    \\
    &\quad +
      \left( \norma{u_o^h}_{\L\infty (\reali^n)}
      + \int_{t_o}^t
      \norma{q^\theta_h(\tau)}_{\L\infty (\reali^n)} (\tau)\d\tau \right)
      \int_{t_o}^t \norma{\nabla (b^\theta_h
      - \div c^\theta) (\tau)}_{\L1 (\reali^n; \reali^n)}\d{s}
    \\
    & \qquad \left.
      \times  \exp \left(
      \int_{t_o}^t
      \norma{\nabla c^\theta (\sigma)}_{\L\infty (\reali^n;\reali^{n \times n})}
      \d\sigma
      \right)
      \int_{t_o}^t
      \norma{(c-\tilde{c}) (s)}_{\L\infty (\reali^n; \reali^n)}\d{s}
      \right\}.
  \end{align*}
  Inserting the result above in~\eqref{eq:11}, by the definitions of
  $b_h^\theta$, $q_h^\theta$ and their properties as stated in
  Lemma~\ref{lem:approx}, we have
  \begin{align*}
    & \norma{u_h (t) - \tilde{u}_h (t)}_{\L1 (\reali^n)}
    \\
    \leq \
    & \exp\left(
      \int_{t_o}^t \max
      \left\{
      \norma{b(\tau)}_{\L\infty (\reali^n)},
      \norma{\tilde b(\tau)}_{\L\infty (\reali^n)}
      \right\}
      \d\tau\right)
      \Biggl\{
      \int_{t_o}^t \norma{(q_h - \tilde q_h) (\tau)}_{\L1 (\reali^n)}\d\tau
    \\
    & + \left(
      \norma{u_o^h}_{\L\infty (\reali^n)}
      + \int_{t_o}^t  \max
      \left\{
      \norma{q(\tau)}_{\L\infty (\reali^n)},
      \norma{\tilde q(\tau)}_{\L\infty (\reali^n)}
      \right\}\d\tau
      \right)
    \\
    & \times
      \int_{t_o}^t
      \norma{(b_h - \tilde{b}_h - \div (c - \tilde{c})) (\tau)}_{\L1 (\reali^n)}
      \d\tau
    \\
    & + \exp\left(
      \int_{t_o}^t \max\left\{
      \norma{\nabla c (s)}_{\L\infty (\reali^n; \reali^{n \times n})},
      \norma{\nabla \tilde{c} (s)}_{\L\infty (\reali^n; \reali^{n \times n})}
      \right\}\d{s}
      \right)
    \\
    & \times
      \int_{t_o}^t  \norma{(c-\tilde{c}) (s)}_{\L\infty (\reali^n; \reali^n)}\d{s}
    \\
    & \quad \times \biggl[
      \int_{\reali^n}\modulo{\nabla u_o^h (y)} \d{y}
      + \int_{t_o}^t \max\left\{
      \norma{\nabla q_h (s)}_{\L1 (\reali^n)}, \,
      \norma{\nabla \tilde q_h (s)}_{\L1 (\reali^n)}
      \right\} \d{s}
    \\
    & \qquad
      +
      \left( \norma{u_o^h}_{\L\infty (\reali^n)} +
      \int_{t_o}^t  \max
      \left\{
      \norma{q(\tau)}_{\L\infty (\reali^n)},
      \norma{\tilde q(\tau)}_{\L\infty (\reali^n)}
      \right\}\d\tau
      \right)
    \\
    & \qquad \quad \left.\left. \times
      \int_{t_o}^t
      \max \left\{
      \norma{\nabla (b_h - \div c) (s)}_{\L1 (\reali^n; \reali^n)},
      \norma{\nabla (\tilde{b}_h - \div \tilde{c}) (s)}_{\L1 (\reali^n; \reali^n)}
      \right\}\d{s}
      \right]\right\} \,.
  \end{align*}
  Let now $h$ tend to $+ \infty$. We have:
  \begin{align*}
    \norma{u_h (t) - \tilde{u}_h (t)}_{\L1 (\reali^n)}
    \to \
    & \norma{u (t) - \tilde{u} (t)}_{\L1 (\reali^n)}
    \\
    \norma{u_o^h}_{\L\infty (\reali^n)}
    \leq \
    &\norma{u_o}_{\L\infty (\reali^n)}
    & &\mbox{by~\eqref{eq:6}}
    \\
    \norma{(q_h - \tilde{q}_h) (\tau)}_{\L1 (\reali^n)}
    \to \
    & \norma{(q - \tilde q) (\tau)}_{\L1 (\reali^n)}
    & &\mbox{by Lemma~\ref{lem:approx}}
    \\
    \norma{(b_h - \tilde{b}_h - \div (c - \tilde{c})) (\tau)}_{\L1 (\reali^n)}
    \to \
    & \norma{
      (b - \tilde{b} - \div (c - \tilde{c})) (\tau)
      }_{\L1 (\reali^n)}
    & &\mbox{by Lemma~\ref{lem:approx}}
    \\
    \norma{ \nabla u_o^h}_{\L1 (\reali^n; \reali^n)}
    \to \
    &\tv(u_o)
    & &\mbox{by~\eqref{eq:6}}
    \\
    \norma{\nabla b_h (s)}_{\L1 (\reali^n; \reali^n)}
    \leq \
    & \tv \left(b(s)\right)
    & &\mbox{by Lemma~\ref{lem:approx}}
    \\
    \norma{\nabla \tilde{b}_h (s)}_{\L1 (\reali^n; \reali^n)}
    \leq \
    & \tv \left(\tilde{b}(s)\right)
    & &\mbox{by Lemma~\ref{lem:approx}}
    \\
    \norma{\nabla q_h (s)}_{\L1 (\reali^n; \reali^n)}
    \leq \
    & \tv \left(q(s)\right)
    & &\mbox{by Lemma~\ref{lem:approx}}
    \\
    \norma{\nabla \tilde{q}_h (s)}_{\L1 (\reali^n; \reali^n)}
    \leq \
    & \tv \left(\tilde{q}(s)\right)
    & & \mbox{by Lemma~\ref{lem:approx}}
  \end{align*}
  Therefore,
  \begin{align}
    \label{eq:3}
    & \norma{u (t) - \tilde{u} (t)}_{\L1 (\reali^n)}
    \\
    \nonumber
    \leq \
    & \exp\left(
      \int_{t_o}^t \max
      \left\{
      \norma{b(\tau)}_{\L\infty (\reali^n)},
      \norma{\tilde b(\tau)}_{\L\infty (\reali^n)}
      \right\}\d\tau\right)
      \Biggl\{
      \int_{t_o}^t \norma{(q - \tilde q) (\tau)}_{\L1 (\reali^n)}\d\tau
    \\
    \nonumber
    & + \left(
      \norma{u_o}_{\L\infty (\reali^n)}
      + \int_{t_o}^t  \max
      \left\{
      \norma{q(\tau)}_{\L\infty (\reali^n)},
      \norma{\tilde q(\tau)}_{\L\infty (\reali^n)}
      \right\}\d\tau
      \right)
    \\ \nonumber
    & \qquad \times
      \int_{t_o}^t
      \norma{(b - \tilde{b} - \div (c - \tilde{c})) (\tau)}_{\L1 (\reali^n)}
      \d\tau
    \\
    \nonumber
    & + \exp\left(
      \int_{t_o}^t \max\left\{
      \norma{\nabla c (s)}_{\L\infty (\reali^n; \reali^{n \times n})},
      \norma{\nabla \tilde{c} (s)}_{\L\infty (\reali^n; \reali^{n \times n})}
      \right\}\d{s}
      \right)
    \\ \nonumber
    & \qquad\times
      \int_{t_o}^t  \norma{(c-\tilde{c}) (s)}_{\L\infty (\reali^n; \reali^n)}\d{s}
    \\
    \nonumber
    &  \times \biggl[
      \tv (u_o)
      + \int_{t_o}^t \max\left\{
      \tv \left(q(s)\right), \,
      \tv \left( \tilde q (s)\right)
      \right\} \d{s}
    \\
    \nonumber
    & \quad
      +
      \left( \norma{u_o}_{\L\infty (\reali^n)} +
      \int_{t_o}^t  \max
      \left\{
      \norma{q(\tau)}_{\L\infty (\reali^n)},
      \norma{\tilde q(\tau)}_{\L\infty (\reali^n)}
      \right\}\d\tau
      \right)
    \\
    \nonumber
    & \quad  \left.\left. \times
      \int_{t_o}^t\!
      \max\left\{ \tv\left(b (s) \right) +
      \norma{\nabla \div c (s)}_{\L1 (\reali^n; \reali^n)},
      \tv\left(\tilde b (s) \right)
      + \norma{\nabla \div \tilde{c} (s)}_{\L1 (\reali^n; \reali^n)}
      \right\}\d{s}
      \right]\!\!\right\}.
  \end{align}
  This completes the proof.
\end{proof}

\subsection{Proof of the Main Result}
\label{subs:Main}

\begin{proofof}{Theorem~\ref{thm:main}}
  Choose an initial datum $(u_o, w_o) \in \mathcal{X}^+$. Define
  $u_0 (t,x) = u_o (x)$ and $w_0 (t,x) = w_o (x)$ for
  $(t,x) \in I \times \reali^n$. Then, construct recursively for
  $i=1, 2, \ldots$ the following sequences of functions:
  \begin{equation}
    \label{eq:15}
    \!\!\!\!\!\!\!\!\!\!\!\!
    \begin{array}{@{}r@{}c@{}l@{}}
      a_{i} (t,x)
      & =
      & g\!\left(t, x, u_{i-1} (t,x), w_{i-1} (t,x)\right)\!;
      \\[4pt]
      b_{i} (t,x)
      & =
      & f\!\left(t, x, w_{i-1} (t,x)\right)\!;
      \\[4pt]
      c_{i} (t,x)
      & =
      & \left(v \left(t, w_{i-1} (t)\right)\right) (x);
    \end{array}
    \begin{array}{@{\,}l@{\,}l@{}l@{}}
      u_{i}
      & \mbox{solves }
      & \left\{
        \begin{array}{@{\,}l@{}}
          \partial_t u_{i}
          +
          \div \left(c_{i} (t,x) u_{i}\right) = b_{i} (t,x) u_{i} + q (t,x)
          \\
          u_{i} (t_o,x) = u_o (x);
        \end{array}
      \right.
      \\
      w_{i}
      & \mbox{solves }
      & \left\{
        \begin{array}{@{\,}l@{}}
          \partial_t w_{i} - \mu \, \Delta w_{i} = a_{i} (t,x) w_{i}
          \\
          w_{i} (t_o,x) = w_o (x).
        \end{array}
      \right.
    \end{array}
    \!\!\!\!\!\!\!\!\!\!\!\!\!\!\!\!\!\!\!\!\!\!\!\!
  \end{equation}
  The existence part of the proof amount to verify that $(u_i,w_i)$ is
  a Cauchy sequence in a suitable complete metric space and that its
  limit solves~\eqref{eq:1}. We divide the proof into several steps.

  \paragraph{Step 0:} For all $i \in \naturali$, $(u_i, w_i)$ is well
  defined and
  \begin{equation}
    \label{eq:18}
    \begin{array}{rl@{\quad\mbox{ and }\quad}l}
      \mbox{for all } t \in I
      & u_i (t) \in\mathcal{U}^+
      &  u_i \in \C{0,1} (I;\L1 (\reali^n; \reali_+)),
      \\
      \mbox{for all } t \in I
      & w_i (t) \in \mathcal{U}^+
      & w_i \in \C0 (I;\L1 (\reali^n; \reali_+)).
    \end{array}
  \end{equation}

  \paragraph{Proof of Step~0:}
  For $i=0$, the thesis holds true due to the choice of the initial
  data and the definition of $u_0$ and $w_0$.  We proceed by
  induction.

  Assume now that the claim holds for $i-1$, with $i \geq 1$. Then,
  $a_i \in \L\infty (I\times\reali^n; \reali)$ for all $t\in I$,
  by~\ref{g} and by the inductive
  hypothesis. Proposition~\ref{prop:para}, Proposition~\ref{prop:PBV}
  and Corollary~\ref{cor:para} hence ensure that $w_i$ is well
  defined, with $w_i (t) \in \mathcal{U}^+$ for all $t\in I$.
  Similarly, $b_i$ satisfies~\ref{b*} by~\ref{f} and $c_i$
  satisfies~\ref{c*} by~\ref{v}. An application of
  Proposition~\ref{prop:hyper} ensures the existence of $u_i$, with
  $u_i (t) \in \mathcal{U}^+$ for all $t \in I$. The time regularity
  of $w_i$ follows from~\ref{it:P:timereg} in
  Proposition~\ref{prop:para} and, for $u_i$, from~\ref{item:uLipt} in
  Proposition~\ref{prop:hyper}.

  \medskip

  \paragraph{Step 1:} For all $i \in \naturali$, for all $t\geq t_o$
  \begin{align}
    \label{eq:19}
    \norma{w_i (t)}_{\L1 (\reali^n)}
    \leq \ &
             \norma{w_o}_{\L1 (\reali^n)} \; e^{K_g (t) \, (t-t_o)},
    &
      \norma{w_i (t)}_{\L\infty (\reali^n)}
      \leq \ &
               \norma{w_o}_{\L\infty (\reali^n)} \; e^{K_g (t) \, (t-t_o)},
  \end{align}
  \begin{equation}
    \label{eq:20}
    \begin{aligned}
      \norma{u_i (t)}_{\L1 (\reali^n)} \leq \ & \left(\norma{u_o}_{\L1
          (\reali^n)} +\norma{q}_{\L1 ([t_o,t]\times\reali^n)}\right)
      \\
      & \times \exp \left[ K_f (t) \, (t-t_o) \left( 1 +
          \norma{w_o}_{\L\infty (\reali^n)} e^{K_g (t) \, (t-t_o)}
        \right) \right],
    \end{aligned}
  \end{equation}
  \begin{equation}
    \label{eq:21}
    \begin{aligned}
      \norma{u_i (t)}_{\L\infty (\reali^n)} \leq \ & \left(
        \norma{u_o}_{\L\infty (\reali^n)} + \norma{q}_{\L1 ([t_o,t];
          \L\infty(\reali^n))}\right)
      \\
      & \times \exp \left[ \left(K_f (t) + K_v (t)\right) (t-t_o)
        \left( 1 + \norma{w_o}_{\L\infty (\reali^n)} e^{K_g (t) \,
            (t-t_o)} \right) \right] .
    \end{aligned}
  \end{equation}

  \noindent (The $\L1$ and $\L\infty$ estimates on $w$ are
  \emph{independent} of $u$. This fact plays a key role throughout, in
  particular in \textbf{Step~6} below.)

  \paragraph{Proof of Step~1:} By~\ref{g} and~\eqref{eq:15}, with
  obvious notation, for all $\tau \in [t_o,t]$,
  \begin{displaymath}
    A_i (\tau)
    : =
    \sup_{\xi \in \reali^n} a_i (\tau,\xi)
    =
    \sup_{\xi \in \reali^n}
    g(\tau,\xi,u_{i-1} (\tau,\xi) , w_{i-1} (\tau,\xi))
    \leq
    K_g(\tau) \leq
    K_g(t) \,.
  \end{displaymath}
  Hence, \eqref{eq:19} follows by~\ref{it:P:priori} in
  Corollary~\ref{cor:para}.

  Proceeding now similarly, using~\eqref{eq:15}, \ref{f}
  and~\eqref{eq:19}, compute for $\tau \in [t_o,t]$,
  \begin{align*}
    \sup_{x \in \reali^n} b_i (\tau,x)
    = \ &
          \sup_{x \in \reali^n}
          f \left(\tau,x, w_{i-1} (\tau,x)\right)
          \leq
          \sup_{x \in \reali^n} K_f (\tau) \left(1+w_{i-1} (\tau,x)\right)
    \\
    \leq \
        & K_f (\tau) \left(1+\norma{w_i (\tau)}_{\L\infty (\reali^n)}\right)
          \leq
          K_f (t)
          \left(1+\norma{w_o}_{\L\infty (\reali^n)} \;
          e^{K_g (t) \, (t-t_o)}\right) \,.
  \end{align*}
  Estimate~\eqref{eq:20} now follows from~\ref{item:uL1} in
  Proposition~\ref{prop:hyper} and~\eqref{eq:15}. Moreover,
  by~\ref{v},
  \begin{displaymath}
    \norma{\div c (\tau)}_{\L\infty (\reali^n; \reali)}
    \leq
    K_v (\tau) \, \norma{w_{i-1} (\tau)}_{\L\infty (\reali^n)}
    \leq
    K_v (t) \, \norma{w_o}_{\L\infty (\reali^n)} \; e^{K_g (t) \, (t-t_o)}.
  \end{displaymath}
  Using now~\ref{item:uLinf} in Proposition~\ref{prop:hyper}
  and~\eqref{eq:15}, the bound~\eqref{eq:21} follows.

  \medskip

  \paragraph{Step 2:} There exists $\mathcal{G} \in \C0 (I;\reali_+)$
  such that for all $t \in I$ and $i \in \naturali$,
  $\tv (w_i (t)) \leq \mathcal{G} (t)$.

  \paragraph{Proof of Step~2:} By the definition of $a_i$ given
  in~\eqref{eq:15}, by~\ref{g} and by~\ref{item:tvwPOS} in
  Corollary~\ref{cor:para} we obtain
  $ \tv (w_i (t)) \leq \mathcal{G} (t)$ where
  \begin{displaymath}
    \mathcal{G} (t) =
    \tv (w_o)
    +
    \frac{2 \, J_n\, \sqrt{t-t_o}}{\sqrt{\mu}} \,
    K_g (t) \,
    \norma{w_o}_{\L\infty (\reali^n)} \,
    e^{K_g (t)\, (t-t_o)}.
  \end{displaymath}

  \medskip

  \paragraph{Step 3:} There exists $\mathcal{F} \in \C0 (I;\reali_+)$
  such that, for all $t \in I$ and all $i \in \naturali$,
  $\tv (u_i (t)) \leq \mathcal{F} (t)$.

  \paragraph{Proof of Step~3:}
  Exploiting the definitions of $b_i$ and $c_i$ given
  in~\eqref{eq:15}, by~\ref{v}, for $\tau \in [t_0,t]$,
  \begin{align*}
    \norma{\nabla c_i (\tau)}_{\L\infty (\reali^n;\reali^{n \times n})}
    \leq \
    & K_v (\tau) \norma{w_{i-1} (\tau)}_{\L\infty (\reali^n)}
      \leq
      K_v (\tau) \, \norma{w_o}_{\L\infty (\reali^n)} \,
      e^{K_g (\tau) \, (\tau-t_o)},
    \\
    \norma{\nabla \div c_i (\tau)}_{\L1 (\reali^n; \reali^n)}
    \leq \
    & C_v \left(\tau, \norma{w_{i-1} (\tau)}_{\L1 (\reali^n)}\right)
      \norma{w_{i-1} (\tau)}_{\L1 (\reali^n)}
    \\
    \leq \
    & C_v \left(
      \tau, \norma{w_o}_{\L1 (\reali^n)} \, e^{K_g (\tau) \, (\tau-t_o)}
      \right)
      \norma{w_o}_{\L1 (\reali^n)} e^{K_g (\tau) \, (\tau-t_o)}.
  \end{align*}
  and 
  by~\ref{f}, \eqref{eq:19} and \textbf{Step~2},
  \begin{align*}
    \tv (b_i (\tau))
    = \
    & \tv \left(f (\tau, \cdot, w_{i-1} (\tau,\cdot))\right)
    \\
    \leq \
    & K_f (\tau)
      \left(
      1
      +
      \norma{w_{i-1} (\tau)}_{\L\infty (\reali^n)}
      +
      \tv \left(w_{i-1} (\tau)\right)
      \right)
    \\
    \leq \
    & K_f (\tau)
      \left(
      1
      +
      \tv (w_o)
      +
      \norma{w_o}_{\L\infty (\reali^n)}
      \left(
      1
      +
      \frac{2 \, J_n\, \sqrt{\tau-t_o}}{\sqrt{\mu}} \,
      K_g (\tau)
      \right)
      e^{K_g (\tau) \, (\tau-t_o)}
      \right) .
  \end{align*}
  Insert the latter estimates above in~\ref{item:utv} of
  Proposition~\ref{prop:hyper} to get
  $\tv (u_i (t)) \leq\mathcal{F} (t)$, where
  \begin{align*}
    \mathcal{F} (t) = \
    & \left(\norma{u_o}_{\L\infty (\reali^n)} + \tv (u_o)
      + \int_{t_o}^t \left(
      \norma{q (\tau)}_{\L\infty (\reali^n)}+ \tv\left(q (\tau)\right)
      \right) \d\tau
      \right)
    \\
    & \times
      \exp\left( \int_{t_o}^t
      \left(
      K_f (\tau) + \norma{w_o}_{\L\infty (\reali^n)}
      \left(K_f (\tau) + K_v (\tau)\right)
      e^{K_g (\tau) (\tau-t_o)}
      \right)
      \d\tau
      \right)
    \\
    & \quad \times
      \left(
      1
      +
      \int_{t_o}^t C_v \left(\tau, \norma{w_o}_{\L1 (\reali^n)}
      e^{K_g (\tau) (\tau-t_o)}\right)
      \norma{w_o}_{\L1 (\reali^n)} \, e^{K_g (\tau) (\tau-t_o)} \d\tau
      \right.
    \\
    &   \qquad
      \left.
      +
      K_f (t) (t-t_o)
      \left(
      1
      +
      \tv (w_o)
      +
      \frac{4\, J_n}{3\, \sqrt{\mu}} \, \sqrt{t-t_o} \,
      K_g (t) \, \norma{w_o}_{\L\infty (\reali^n)} \,
      e^{K_g (t)(t-t_o)}
      \right)
      \right)
  \end{align*}
  concluding the proof of \textbf{Step~3}.

  \medskip

  Observe for later use that, due to~\ref{f}, \ref{g} and~\ref{v}, on
  a bounded time interval $[t_o,T]$
  \begin{equation}
    \label{eq:4}
    a_{i+1} - a_i, \
    b_{i+1} - b_i, \
    \div (c_{i+1} - c_i)
    \in  \L1 ([t_o,T] \times\reali^n;\reali) \,.
  \end{equation}

  \medskip

  \paragraph{Step 4:} Referring to~\eqref{eq:X}, \eqref{eq:normX},
  \textbf{Step~2} and~\textbf{Step~3}, consider the complete metric
  space
  \begin{align}
    \nonumber
    & \mathcal{Y}_T
      =
      \left\{
      (u,w) \in \C0 ([t_o,T]; \mathcal{X}^+) \colon
      \tv (u (t))\leq \mathcal{F} (t) \mbox{ and }
      \tv (w (t)) \leq \mathcal{G} (t)
      \mbox{ for all } t \in [t_o,T]
      \right\},
    \\
    \label{eq:Y}
    & d \left((u_1,w_1), (u_2,w_2)\right)
      =
      \sup_{t\in[t_o,T]}
      \norma{(u_1 (t) -u_2 (t), \, w_1 (t) - w_2 (t))}_{\mathcal{X}}.
  \end{align}
  Moreover, for $r>0$ introduce the following subset of
  $\mathcal{X}^+$:
  \begin{equation}
    \label{eq:38}
    \mathcal{X}^+_r = \left\{
      (u,w) \in \mathcal{X}^+ \colon
      \begin{array}{r@{\,}c@{\,}l@{\qquad}r@{\,}c@{\,}l@{\qquad}r@{\,}c@{\,}l}
        \norma{u}_{\L\infty (\reali^n)}
        & \leq & r,
        & \tv (u)
        & \leq
        & r,
        \\
        \norma{w}_{\L\infty (\reali^n)}
        & \leq
               & r,
        & \norma{w}_{\L1 (\reali^n)}
        & \leq
        & r,
        & \tv (w)
               & \leq
        & r
      \end{array}
    \right\} .
  \end{equation}
  Then, given $(u_o, w_o) \in \mathcal{X}^+_r$, there exists a
  continuous function $\mathcal{K}_r \colon [t_o, T] \to \reali_+$,
  for a suitable $T \in I$ with $T>t_o$, such that for all
  $i\in \naturali$
  \begin{equation}
    \label{eq:24}
    d \left((u_{i+1},w_{i+1}), (u_i, w_i)\right)  \leq
    \mathcal{K}_r (T) \,(T-t_o) \; d \! \left((u_i,w_i), (u_{i-1}, w_{i-1})\right).
  \end{equation}

  \paragraph{Proof of Step~4:} In the following, we make use of the
  bounds~\eqref{eq:19}--\eqref{eq:21}. Start from~\ref{it:P:stab} in
  Corollary~\ref{cor:para}: for all $t \in [t_o,T]$,
  using~\eqref{eq:15} and~\ref{g}, we obtain
  \begin{align}
    \nonumber
    & \norma{w_{i+1} (t) - w_i (t)}_{\L1 (\reali^n)}
    \\ \nonumber
    \leq \
    &  K_g (t) \, (t-t_o) \, \norma{w_o}_{\L\infty (\reali^n)}
      e^{2 \, (t-t_o) K_g(t)}
      \sup_{\tau \in [t_o,t]} \norma{ (u_i (\tau) -u_{i-1}
      (\tau), \, w_i (\tau) - w_{i-1} (\tau))  }_{\mathcal{X}}
    \\
    \label{eq:22}
    \leq \
    & \mathcal{K}^w_r (T) \, (T-t_o) \;
      d\!\left((u_i,w_i), (u_{i-1}, w_{i-1})\right),
  \end{align}
  with
  \begin{equation}
    \label{eq:Kw}
    \mathcal{K}_r^w (T) = r \, K_g (T) \, e^{2 \, (T-t_o) K_g(T)}.
  \end{equation}
  Now consider~\ref{item:ustab} in Proposition~\ref{prop:hyper}:
  by~\ref{v} and~\ref{f}, setting
  \begin{align}
    \label{eq:O1tw}
    \mathcal{\tilde{O}}_1 (t) = \
    & \exp\left( K_f (t) (t-t_o)
      +
      \norma{w_o}_{\L\infty (\reali^n)} \,  (t-t_o)
      \left(K_f (t) + K_v (t)\right) e^{K_g (t) \, (t-t_o)}
      \right)
    \\
    \nonumber
    & \times \bigg[
      1
      +
      (t-t_o) \, C_v (t, \norma{w_o}_{\L1 (\reali^n)} e^{K_g (t) (t-t_o)})
      \norma{w_o}_{\L1 (\reali^n)} e^{K_g (t) (t-t_o)}
    \\
    \nonumber
    & \quad
      +
      K_f (t) (t-t_o)
      \left(
      1
      +
      \tv (w_o)
      +
      \frac{4\, J_n}{3\, \sqrt{\mu}} \, \sqrt{t-t_o} \,
      \norma{w_o}_{\L\infty (\reali^n)} \,
      K_g (t) \, e^{K_g (t) (t-t_o)}
      \right)
      \bigg],
    \\
    \label{eq:O2tw}
    \mathcal{\tilde{O}}_2 (t) = \
    & \exp
      \left(
      K_f (t) \, (t-t_o)
      \left(
      1
      +
      \norma{w_o}_{\L\infty (\reali^n)} e^{K_g (t) (t-t_o)}\right)
      \right),
  \end{align}
  we get
  \begin{align}
    & \nonumber
      \norma{u_{i+1} (t) - u_i (t)}_{\L1 (\reali^n)}
    \\ \nonumber
    \leq \
    &\left[ \mathcal{\tilde{O}}_1 (t)
      (t-t_o)
      \left(
      \norma{u_o}_{\L\infty (\reali^n)}
      +
      \tv (u_o)
      +
      \int_{t_o}^t \left(
      \norma{q (\tau)}_{\L\infty (\reali^n)}
      +
      \tv\left(q (\tau)\right)
      \d\tau
      \right)
      \right)
      K_v (t) \right.
    \\ \nonumber
    & \left. + \mathcal{\tilde{O}}_2 (t) (t-t_o)\!
      \left(\!
      \norma{u_o}_{\L\infty (\reali^n)}
      +
      \int_{t_o}^t \norma{q (\tau)}_{\L\infty (\reali^n)}
      \d\tau\!
      \right)\!\!
      \left(\!
      K_f (t)
      +
      C_v\left(t, \norma{w_o}_{\L\infty} e^{K_g (t) (t-t_o) }\right)\right)\!\right]
    \\
    \nonumber
    & \times \sup_{\tau \in [t_o,t]}
      \norma{w_i (\tau) - w_{i-1} (\tau)}_{\L1 (\reali^n)}
    \\
    \label{eq:23}
    \leq \
    & \mathcal{K}_r^u (T) \, (T-t_o)
      \sup_{\tau \in [t_o,T]} \norma{w_i (\tau) - w_{i-1} (\tau)}_{\L1 (\reali^n)},
  \end{align}
  with
  \begin{align}
    \nonumber
    \mathcal{K}_r^u (T) = \
    & \left(
      r
      +
      \int_{t_o}^T \norma{q (\tau)}_{\L\infty (\reali^n)} \d\tau
      \right)
      \exp\left(K_f(T) (T-t_o) \left(1 + r \, e^{K_g (T) (T-t_o)}\right)\right)
    \\
    \nonumber
    & \qquad
      \times
      \left[
      K_f (T) + C_v \left(T,r \,  e^{K_g (T) (T-t_o) }\right)
      \right]
    \\
    \nonumber
    & +
      \left(
      2 \, r
      +
      \int_{t_o}^T
      \left(
      \norma{q (\tau)_{\L\infty (\reali^n)}}
      +
      \tv\left(q (\tau)\right)  \right)
      \d\tau
      \right)
    \\
    \label{eq:Ku}
    & \times \exp\left(K_f(T) (T-t_o) \left(1 + r \, e^{K_g (T) (T-t_o)}\right)\right)
    \\ \nonumber
    & \quad \times K_v (T)
      \exp
      \left(K_v (T) \, (T-t_o) \, r \, e^{K_g (T) (T-t_o)}\right)
    \\ \nonumber
    & \quad
      \times \biggl(
      1 +
      r \, (T-t_o) \, C_v (T, r \, e^{K_g (T) (T-t_o)}) e^{K_g (T) (T-t_o)}
    \\ \nonumber
    & \qquad
      \left.
      +
      (T-t_o) \, K_f (T)
      \left(
      1
      +
      r
      +
      \frac{4\, J_n}{3\, \sqrt{\mu}} \, K_g (T) \, \sqrt{T-t_o} \, r
      \, e^{K_g (T) (T-t_o)}
      \right)
      \right).
  \end{align}
  Thus, collecting together~\eqref{eq:22} and~\eqref{eq:23},
  \begin{align*}
    d \left((u_{i+1},w_{i+1}), (u_i,w_i)\right)
    = \
    & \sup_{t \in  [0,T]}\left(\norma{u_{i+1} (t) -u_i (t)}_{\L1 (\reali^n)}
      +\norma{w_{i+1} (t) -w_i (t)}_{\L1 (\reali^n)}\right)
    \\
    \leq \
    & \left(\mathcal{K}_r^u (T) + \mathcal{K}_r^w (T)\right)\, (T-t_o) \,
      d\left((u_i,w_i), (u_{i-1},w_{i-1})\right).
  \end{align*}
  This proves~\eqref{eq:24}, with
  $\mathcal{K}_r (T) =\mathcal{K}_r^u (T) + \mathcal{K}_r^w (T)$
  and~\textbf{Step~4} is completed.

  \medskip

  \paragraph{Step 5:} For any $r > 0$, there exists a $T_r > 0$ such
  that for all $(u_o, w_o) \in \mathcal{X}_r^+$, the sequence
  $(u_i, w_i)$ converges in $\mathcal{Y}_{T_r}$ to a $(u_*,w_*)$
  solving~\eqref{eq:1} in the sense of Definition~\ref{def:sol}.

  \paragraph{Proof of Step~5:} Choose $T_r > t_o$ such that
  $\mathcal{K}_r(T_r) \, (T_r-t_o) < 1$. Thanks to~\eqref{eq:24}, the
  sequence $(u_i , w_i )$ defined through~\eqref{eq:15} is a Cauchy
  sequence and converges in the complete metric space
  $(\mathcal{Y}_{T_r},d)$ defined in~\eqref{eq:Y}. Call $(u_*, w_*)$
  the limit. Clearly, $u_* \in \C0 ([t_o, T_r]; \mathcal{U}^+)$ and
  $ w_* \in \C0 ([t_o,T_r]; \mathcal{U}^+)$. It remains to prove that
  $(u_*, w_*)$ is a solution to~\eqref{eq:1} in the sense of
  Definition~\ref{def:sol}. By Lemma~\ref{lem:paraSol} and
  Lemma~\ref{lem:hypSol} it is sufficient to prove that $u_*$ is a
  weak solution to~\eqref{eq:hyp} and $w_*$ is a weak solution
  to~\eqref{eq:par} with
  \begin{displaymath}
    a (t,x)
    =
    g\left(t, x, u_* (t,x), w_* (t,x)\right) \,,
    \quad
    b (t,x)
    =
    f\left(t, x, w_* (t,x)\right) \,,
    \quad
    c (t,x)
    =
    \left(v\left(t, w_* (t)\right)\right) \!(x) .
  \end{displaymath}
  The initial condition is satisfied: $(u_*,w_* )(0) = (u_o ,
  w_o)$. Using the weak formulations~\eqref{eq:Hsol}
  and~\eqref{eq:Psol}, applying the Dominated Convergence Theorem,
  thanks to~\ref{f} and~\ref{g}, we obtain that $(u_*, w_*)$
  solves~\eqref{eq:1} on $[t_o, T_r]$, with initial datum
  $(u_o , w_o )$, in the sense of Definition~\ref{def:sol}.

  \medskip

  \paragraph{Step~6:} The solution constructed above can be uniquely
  extended to all $I$.

  \paragraph{Proof of Step~6:} The uniform continuity in time of
  $(u_*,w_*)$ on $[t_o, T_r]$ ensures that
  $\left(u_* (T_r), w_* (T_r)\right) = \lim_{t \to T_r -} \left(u_*
    (t), w_* (t)\right)$ is in $\mathcal{X}^+$. The above results can
  be iteratively applied, proving that $(u_*,w_*)$ can be uniquely
  extended to a maximal time interval $\left[t_o, T_*\right[$.

  The $\L1$ and $\L\infty$ bounds in~\eqref{eq:19}, together with the
  $\BV$ bound in~\textbf{Step~2}, ensure that the limit
  $\lim_{t \to T_*-} w_* (t)$ exists and is in $\mathcal{U}^+$, so
  that we can define $w_* (T_*) = \lim_{t \to T_*-} w_*
  (t)$. Similarly, Proposition~\ref{prop:hyper}, allows to uniquely
  extend $u_*$ in $T_*$, setting
  $u_* (T_*) = \lim_{t \to T_*-} u_* (t)$ with
  $u_* (T_*) \in \mathcal{U}^+$. A further application of the steps
  above then allows to further prolong $(u_*, w_*)$ beyond time $T_*$,
  unless $T_* = \sup I$, completing the proof of this step.

  \medskip

  \paragraph{Step~7:} Let $r>0$. Given
  $(u_o,w_o), \, (\tilde u_o, \tilde w_o) \in \mathcal{X}^+_r$, call
  $(u,w)$ and $(\tilde u, \tilde w)$ the corresponding solutions
  to~\eqref{eq:1}. Then, for all $t \in I$, \eqref{eq:37} holds, with
  $\mathcal{C}_o$ defined in~\eqref{eq:36}.

  \paragraph{Proof of Step~7: } Define for
  $(t,x) \in I \times \reali^n$ the following functions
  \begin{align}
    \nonumber
    a (t,x) = \
    & g \left(t,x,u (t,x), w (t,x) \right),
    &
      \tilde a (t,x) = \
    & g \left(t,x, \tilde u (t,x), \tilde w (t,x) \right),
    \\
    \label{eq:utile}
    b (t,x) = \
    & f\left(t, x, w (t,x) \right),
    &
      \tilde b (t,x) = \
    & f \left(t, x, \tilde w (t,x) \right),
    \\
    \nonumber
    c (t,x) = \
    & \left(v\left(t, w (t)\right)\right) (x),
    &
      \tilde c (t,x) = \
    & \left(v\left(t, \tilde w (t)\right)\right) (x).
  \end{align}
  Let $\hat w$ be the solution to~\eqref{eq:par} with $a$ in the
  source term and initial datum $\tilde w_o$, and let $\hat u$ be the
  solution to~\eqref{eq:hyp} with coefficients $b, c$ and initial
  datum $\tilde u_o$. More precisely,
  \begin{equation}
    \label{eq:par_hyp}
    \left\{
      \begin{array}{l}
        \partial_t \hat w - \mu \, \Delta \hat w = a (t,x) \,\hat w
        \\
        \hat w(t_o,x) = \tilde w_o (x)
      \end{array}
    \right.
    \quad \mbox{ and } \quad
    \left\{
      \begin{array}{l}
        \partial_t \hat u +\div (c (t,x) \, \hat u)
        = b (t,x) \, \hat u + q (t,x)
        \\
        \hat u(t_o,x) = \tilde u_o (x).
      \end{array}
    \right.
  \end{equation}
  By~\eqref{eq:normX}, we need to compute
  \begin{align}
    \nonumber
    \norma{(u (t), w(t))  - (\tilde u (t), \tilde w (t))}_{\mathcal{X}}
    = \
    &   \norma{u (t) - \tilde u (t)}_{\L1 (\reali^n)}
      +\norma{w (t) - \tilde w (t)}_{\L1 (\reali^n)}
    \\
    \label{eq:25}
    \leq \
    &  \norma{u (t) - \hat u (t)}_{\L1 (\reali^n)}
      + \norma{\hat u (t) - \tilde u (t)}_{\L1 (\reali^n)}
    \\
    \label{eq:26}
    &+ \norma{w (t) - \hat w (t)}_{\L1 (\reali^n)}
      +\norma{\hat w (t) - \tilde w (t)}_{\L1 (\reali^n)}.
  \end{align}
  Compute each term in~\eqref{eq:25} separately. The first one is the
  $\L1$--distance between solutions to balance laws of the
  type~\eqref{eq:hyp} with different initial
  data. Exploiting~\eqref{eq:solCara} for the solution to these
  balance laws and the bounds obtained in the proof
  of~\textbf{Step~1}, we get
  \begin{equation}
    \label{eq:25a}
    \!\!\!\!
    \norma{u (t) - \hat u (t)}_{\L1 (\reali^n)}
    {\leq}  \norma{u_o - \tilde u_o}_{\L1 (\reali^n)}
    \exp\left[\!
      K_f (t) \, (t-t_o)
      \left(1 + \norma{w_o}_{\L\infty (\reali^n)} \, e^{K_g (t) (t-t_o)}\right)
      \!\right]\!.
  \end{equation}
  The second term in~\eqref{eq:25} is the $\L1$--distance between
  solutions to balance laws of the type~\eqref{eq:hyp} with different
  coefficients $b$, $c$ and same initial datum.  Exploiting the
  computations in the proof of \textbf{Step~4}, as well
  as~\ref{item:ustab} in Proposition~\ref{prop:hyper}, we get
  \begin{align}
    \nonumber
    & \norma{\hat u (t) - \tilde u (t)}_{\L1 (\reali^n)}
    \\
    \nonumber
    \leq \
    & \left\{ \mathcal{\hat{O}}_1 (t,r) \left(
      \norma{\tilde u_o}_{\L\infty (\reali^n)}
      + \tv (\tilde u_o)
      + \int_{t_o}^t
      \left(
      \norma{q (\tau)}_{\L\infty (\reali^n)}
      +
      \tv\left(q (\tau) \right)
      \right) \d\tau
      \right)  K_v (t) \right.
    \\
    \nonumber
    & {+} \mathcal{\hat{O}}_2 (t,r) \!
      \left( \!
      \norma{\tilde u_o}_{\L\infty (\reali^n)}
      {+} \int_{t_o}^t  \norma{q (\tau)}_{\L\infty (\reali^n)} \d\tau
      \right)
          \left(\!
          K_f (t)
          {+} C_v \left( t, \norma{\tilde w_o}_{\L\infty (\reali^n)}
          e^{K_g (t) (t-t_o)}\right)\right)\Biggr\}
    \\
    \label{eq:25b}
    &
      \times
      \int_{t_o}^t \norma{w (\tau) - \tilde w (\tau)}_{\L1 (\reali^n)}\d\tau,
  \end{align}
  with
  \begin{align}
    \nonumber
    \mathcal{\hat{O}}_1 (t,r) = \
    & \exp\left( K_f (t) \, (t-t_o)
      +
      r\, (t-t_o) \,  e^{K_g (t) (t-t_o)}
      \left(K_f (t) + K_v (t)\right)
      \right)
    \\
    \label{eq:01r}
    & \times \left[ 1
      +
      (t-t_o) \, C_v (t, r \,  e^{K_g (t) (t-t_o)})\,
      r \, e^{K_g (t) (t-t_o)}\right.
    \\
    \nonumber
    & \qquad \left.
      +
      K_f (t) \, (t-t_o)
      \left(
      1
      +
      r
      +
      \frac{4\,J_n}{3\, \sqrt{\mu}} \,
      r\, K_g (t) \, \sqrt{t-t_o} \, e^{K_g (t) (t-t_o)}
      \right)
      \right],
    \\
    \label{eq:O2r}
    \mathcal{\hat{O}}_2 (t,r) = \
    & \exp
      \left(
      K_f (t) \, (t-t_o)
      \left( 1 + r \,  e^{K_g (t) (t-t_o)}\right)
      \right).
  \end{align}

  The first term in~\eqref{eq:26} is the $\L1$--distance between
  solutions to equations of the type~\eqref{eq:par} with different
  initial data. Since $\mathcal{P}$ as defined in
  Proposition~\ref{prop:para} is linear, by \textbf{Step~1} we obtain
  \begin{equation}
    \label{eq:26a}
    \norma{w (t) - \hat w (t)}_{\L1 (\reali^n)}
    \leq
    \norma{w_o - \tilde w_o}_{\L1 (\reali^n)} \,
    \exp\left(K_g (t) (t-t_o)\right).
  \end{equation}
  The second term in~\eqref{eq:26} is the $\L1$--distance between
  solutions to the parabolic equation~\eqref{eq:par} with different
  coefficients in the source term and the same initial
  datum. Exploiting the computations in the proof of \textbf{Step~4},
  as well~\ref{it:P:stab} in Corollary~\ref{cor:para}, we get
  \begin{equation}
    \label{eq:26b}
    \norma{\hat w (t) - \tilde w (t)}_{\L1 (\reali^n)}
    \leq
    \norma{\tilde w_o}_{\L\infty (\reali^n)} K_g (t) \, e^{2 \,K_g (t) (t-t_o)}
    \int_{t_o}^t
    \norma{
      \left(
        u (\tau) - \tilde u (\tau), \, w (\tau) - \tilde w (\tau)
      \right)}_{\mathcal{X}}
    \d\tau.
  \end{equation}
  Hence, \eqref{eq:25a}, \eqref{eq:25b}, \eqref{eq:26a}
  and~\eqref{eq:26b} yield
  \begin{align*}
    \norma{(u (t), w(t))  - (\tilde u (t), \tilde w (t))}_{\mathcal{X}}
    \leq \
    & \mathcal{K}_1 (t, r) \,
      \norma{(u_o, w_o)  - (\tilde u_o, \tilde w_o)}_{\mathcal{X}}
    \\
    & \!\! + \mathcal{K}_2 (t, r)
      \int_{t_o}^t \left(
      \norma{\left(
      u (\tau) - \tilde u (\tau), \, w (\tau) - \tilde w (\tau)
      \right)}_{\mathcal{X}}
      \right) \d\tau,
  \end{align*}
  where we set
  \begin{align*}
    \mathcal{K}_1 (t,r) = \
    & \exp
      \left( \max
      \left\{
      K_f (t) \, (t-t_o)
      \left(1 + r \, e^{K_g (t) (t-t_o)}\right), \,
      e^{K_g (t) (t-t_o)}
      \right\}
      \right),
    \\
    \mathcal{K}_ 2(t, r) = \
    &
      \mathcal{\hat{O}}_1 (t,r) \left(
      2\, r
      +
      \int_{t_o}^t
      \left(
      \norma{q (\tau)}_{\L\infty (\reali^n)}
      +
      \tv\left(q (\tau) \right)
      \right) \d\tau
      \right)  K_v (t)
    \\
    & +
      \mathcal{\hat{O}}_2 (t,r) \left(
      r
      +
      \int_{t_o}^t  \norma{q (\tau)}_{\L\infty (\reali^n)} \d\tau
      \right)
      \left(
      K_f (t)
      + C_v \left( t, r \, e^{K_g (t) (t-t_o)}\right)
      \right)
    \\
    &  +
      r \, K_g (t) \, e^{2 \, K_g (t) (t-t_o)}.
  \end{align*}
  An application of Gronwall Lemma yields:
  \begin{displaymath}
    \norma{(u (t), w(t)) - (\tilde u (t), \tilde w
      (t))}_{\mathcal{X}}
    \leq \ \norma{(u_o, w_o) - (\tilde u_o, \tilde
      w_o)}_{\mathcal{X}} \int_{t_o}^t \mathcal{K}_1 (s, r) \,
    \exp\left( \int_s^t \mathcal{K}_2 (\tau, r ) \d\tau \right)
    \d{s} ,
  \end{displaymath}
  proving \textbf{Step~7} with
  \begin{equation}
    \label{eq:36}
    \mathcal{C}_o (t, r) = \int_{t_o}^t \mathcal{K}_1 (s, r) \,
    \exp\left( \int_s^t \mathcal{K}_2 (\tau, r ) \d\tau \right) \d{s} .
  \end{equation}

  \medskip

  \paragraph{Step 8:} Given $q, \, \tilde q$ satisfying~\ref{q*}, call
  $(u,w)$ and $(\tilde u, \tilde w)$ the solutions to~\eqref{eq:1}
  with the same initial datum $(u_o, w_o) \in \mathcal{X}_r^+$. Then,
  for all $t\in I$, \eqref{eq:27} holds with $\mathcal{C}_q$ defined
  in~\eqref{eq:33}.

  \paragraph{Proof of Step 8:}
  Define for $(t,x) \in I \times \reali^n$ the functions
  $a, \tilde a, \, b, \, \tilde b, \, c, \, \tilde c$ as
  in~\eqref{eq:utile}.

  The $\L1$ distance between $w(t)$ and $\tilde w (t)$ can be computed
  as in~\eqref{eq:26b}, leading to
  \begin{equation}
    \label{eq:28}
    \norma{w (t)- \tilde w (t)}_{\L1 (\reali^n)}
    \leq
    \norma{w_o}_{\L\infty (\reali^n)} K_g (t) \, e^{2 \, K_g (t) (t-t_o)}
    \int_{t_o}^t
    \norma{
      (u - \tilde u, \, w - \tilde w)(\tau)}_{\mathcal{X}}
    \d\tau.
  \end{equation}
  To compute the $\L1$ distance between $u (t)$ and $\tilde u (t)$, we
  exploit~\ref{item:ustab} in Proposition~\ref{prop:hyper} and the
  computations in the proofs of \textbf{Step~4} and \textbf{Step~7},
  to get
  \begin{align}
    \nonumber
    &  \norma{u (t)- \tilde u (t)}_{\L1 (\reali^n)}
    \\ \nonumber
    \leq \
    & \mathcal{\hat{O}}_1 (t,r)
      \biggl[ \norma{u_o}_{\L\infty (\reali^n)}
      +
      \tv (u_o)
    \\ \nonumber
    & \qquad \left.+ \int_{t_o}^t \!\!
      \left(
      \max
      \left\{
      \norma{q (\tau)}_{\L\infty (\reali^n)}, \,
      \norma{\tilde q (\tau)}_{\L\infty (\reali^n)}
      \right\}
      + \max\left\{\tv\left( q(\tau)\right), \,
      \tv\left(\tilde q (\tau)\right)\right\}\right) \d\tau
      \right]
    \\ \nonumber
    & \times
      K_v (t)
      \int_{t_o}^t \norma{w (\tau) - \tilde w (\tau)}_{\L1 (\reali^n)} \d\tau
    \\ \nonumber
    & + \mathcal{\hat {O}}_2 (t,r)
      \left(
      \norma{u_o}_{\L\infty (\reali^n)}
      +
      \int_{t_o}^t
      \max
      \left\{
      \norma{q (\tau)}_{\L\infty (\reali^n)}, \,
      \norma{\tilde q (\tau)}_{\L\infty (\reali^n)}
      \right\}
      \d\tau \right)
    \\
    \nonumber
    & \times
      \left(
      K_f (t)
      +
      C_v\left(t, \norma{w_o}_{\L\infty} e^{K_g (t) (t-t_o)}\right)
      \right)
    \\ \label{eq:29}
    & \times
      \int_{t_o}^t \norma{w_i (\tau) - w_{i-1} (\tau)}_{\L1 (\reali^n)} \d\tau
      +
      \mathcal{\hat{O}}_2 (t,r) \,
      \norma{q - \tilde q}_{\L1 ([t_o,t] \times \reali^n)},
  \end{align}
  where $\mathcal{\hat{O}}_1 (t,r)$ and $\mathcal{\hat{O}}_2 (t,r)$
  are as in~\eqref{eq:01r}--\eqref{eq:O2r}. Collecting
  together~\eqref{eq:28} and~\eqref{eq:29} and an application of
  Gronwall Lemma completes the proof of \textbf{Step~8} with
  \begin{align}
    \label{eq:33}
    \mathcal{C}_q (t,r) = \
    &
      \mathcal{\hat{O}}_2 (t,r)
      \int_{t_o}^t
      \exp
      \int_s^t
      \left\{
      r \, K_g (\tau) \, e^{2 \, K_g (\tau) (\tau-t_o)}
      +  K_v (\tau) \,
      \mathcal{\hat{O}}_1 (\tau,r)
      \biggl[
      2\,r
      \right.
    \\
    \nonumber
    & \qquad \left.
      +\! \int_{t_o}^\tau \!\!
      \left[
      \max \left\{
      \norma{q (\sigma)}_{\L\infty (\reali^n)},
      \norma{\tilde q (\sigma)}_{\L\infty (\reali^n)}
      \!
      \right\}
      \!+\!
      \max \left\{\!
      \tv \left( q(\sigma)\right),
      \tv \left(\tilde q (\sigma)
      \right)\!
      \right\}\!
      \right]\!
      \d\sigma\!
      \right]
    \\
    \nonumber
    & \qquad
      + \mathcal{\hat{O}}_2 (\tau,r)
      \left[
      r
      + \int_{t_o}^\tau
      \max
      \left\{
      \norma{q (\sigma)}_{\L\infty (\reali^n)}, \,
      \norma{\tilde q (\sigma)}_{\L\infty (\reali^n)}
      \right\}
      \d\sigma
      \right]
    \\
    \nonumber
    & \qquad\quad  \times
      \left[ K_f (\tau)
      +
      C_v\left(\tau, r e^{K_g (\tau) (\tau-t_o) }\right)\right]
      \Biggr\}
      \d\tau \d{s}.
  \end{align}
\end{proofof}

\paragraph{Acknowledgement:}
The authors were partly supported by the PRIN~2015 project
\emph{Hyperbolic Systems of Conservation Laws and Fluid Dynamics:
  Analysis and Applications}. The GNAMPA~2018 project
\emph{Conservation Laws: Hyperbolic Games, Vehicular Traffic and Fluid
  dynamics} is also acknowledged. The second author acknowledges the
support of the Lorentz Center. The \emph{IBM Power Systems Academic
  Initiative} substantially contributed to the numerical integrations.

{ \small

  \bibliography{ColomboRossi}

  \bibliographystyle{abbrv}

}

\end{document}